\documentclass[11pt]{article}

\usepackage[centertags]{amsmath}
\usepackage{amsfonts}
\usepackage{amssymb}
\usepackage{amsthm}
\usepackage{newlfont}
\usepackage{easybmat}

\newcommand{\blockmatrix}[3]{\begin{minipage}[t][#2][c]{#1}\center#3\end{minipage}}
\newtheorem{thm}{Theorem}[section]
\newtheorem*{thm*}{Theorem}
\newtheorem{cor}[thm]{Corollary}

\newtheorem{lem}[thm]{Lemma}

\newtheorem{prop}[thm]{Proposition}
\newtheorem*{prop*}{Proposition}

\newtheorem*{conj*}{Conjecture}

\newtheorem*{dfn*}{Definition}
\theoremstyle{definition}
\newtheorem{rem}[thm]{\textbf{Remark}}
\newtheorem*{rmk*}{Remark}
\newtheorem*{fact*}{Fact}

\theoremstyle{proof}

\newcommand{\norm}[1]{\left\Vert#1\right\Vert}

\newcommand{\abs}[1]{\left\vert#1\right\vert}
\newcommand{\set}[1]{\left\{#1\right\}}
\newcommand{\brac}[1]{\left(#1\right)}
\newcommand{\scalar}[1]{\left \langle #1 \right \rangle}

\newcommand{\Real}{\mathbb{R}}

\newcommand{\eps}{\varepsilon}

\newcommand{\dist}{\textrm{dist}}
\newcommand{\vol}{\textrm{Vol}}
\newcommand{\Var}{\mathbb V\textrm{ar}}
\newcommand{\Cov}{\mathbb C\textrm{ov}}
\newcommand{\Ent}{\mathbb E\textrm{nt}}
\newcommand{\E}{\mathbb E}
\renewcommand{\P}{\mathbb P}
\newcommand{\n}{\eta}

\newlength{\defbaselineskip}
\numberwithin{equation}{section}

\begin{document}

\title{Interpolating Thin-Shell and Sharp Large-Deviation Estimates For Isotropic Log-Concave Measures}
\author{Olivier Gu\'edon\textsuperscript{1} and Emanuel Milman\textsuperscript{2}}

\date{}

\footnotetext[1]{
Universit\'e Paris-Est Marne La Vall\'ee,
Laboratoire d'Analyse et de Math\'ematiques Appliqu\'ees.
5, Bd Descartes, Champs sur Marne 77454, Marne La Vall\'ee, C\'edex 2, France. Email: olivier.guedon@univ-mlv.fr.}

\footnotetext[2]{Department of Mathematics,
Technion - Israel Institute of Technology, Haifa 32000, Israel. Supported by ISF and the Taub Foundation (Landau Fellow).
Email: emilman@tx.technion.ac.il.}

\maketitle

\begin{abstract}
Given an isotropic random vector $X$ with log-concave density in Euclidean space $\Real^n$, we study the concentration properties of $|X|$ on all scales, both above and below its expectation. We show in particular that:
\[
\P(\abs{|X| -\sqrt{n}} \geq t \sqrt{n}) \leq C \exp(-c n^{\frac{1}{2}} \min(t^3,t)) \;\;\; \forall t \geq 0 ~,
\]
for some universal constants $c,C>0$. This improves the best known deviation results on the thin-shell and mesoscopic scales due to Fleury and Klartag, respectively, and recovers the sharp large-deviation estimate
of Paouris. Another new feature of our estimate is that it improves when $X$ is $\psi_\alpha$ ($\alpha \in (1,2]$), in precise agreement with Paouris' estimates. The upper bound on the thin-shell width $\sqrt{\Var(|X|)}$ we obtain is of the order of $n^{1/3}$, and improves down to $n^{1/4}$ when $X$ is $\psi_2$. Our estimates thus continuously interpolate between a new best known thin-shell estimate and the sharp large-deviation estimate of Paouris.
As a consequence, a new best known bound on the Cheeger isoperimetric constant appearing in a conjecture of Kannan--Lov\'asz--Simonovits is deduced.
\end{abstract}

\section{Introduction}

Let a Euclidean norm $\abs{\cdot}$ on $\Real^n$ be fixed. This work is dedicated to quantitative concentration properties of $|X|$, where $X$ is an isotropic random vector in $\Real^n$ with log-concave density. Recall that a random vector $X$ in $\Real^n$ (and its density) is called isotropic if $\E X = 0$ and $\E X \otimes X = Id$, i.e. its barycenter is at the origin and its covariance matrix is equal to the identity one. For such an $X$, if $A \in M_{n}(\Real)$ denotes an $n$ by $n$ matrix, observe that $\E |A X|^2 = \norm{A}_{HS}^2$, where $\norm{A}_{HS} = \sqrt{\sum_{i,j} A_{i,j}^2}$ denotes the Hilbert--Schmidt norm of $A$.
Here and throughout we use $\E$ to denote expectation, $\P$ to denote probability, and $\Var$ to denote variance. A function $g: \Real^n \rightarrow \Real_+$ is called log-concave if $-\log g : \Real^n \rightarrow \Real \cup \set{+\infty}$ is convex. Throughout this work, $C$,$c$,$c_2$,$C'$, etc. denote universal positive numeric constants, independent of any other parameter and in particular the dimension $n$, whose value may change from one occurrence to the next.

It was conjectured by Anttila, Ball and Perissinaki \cite{ABP} that $|X|$ is concentrated around its expectation significantly more than suggested by the trivial bound $\Var|X| \leq \E|X|^2 = n$. Namely, they conjectured that  there exists a sequence $\set{\eps_n}$ decreasing to $0$ with the dimension $n$, so that $X$ is concentrated within a ``thin shell'' of relative width $2 \eps_n$ around the (approximately) expected Euclidean norm of $\sqrt{n}$:
\begin{equation} \label{eq:ABP-conc}
\P(\abs{|X| - \sqrt{n}} \geq \eps_n \sqrt{n}) \leq \eps_n ~.
\end{equation}
Their conjecture was mainly motivated by the Central Limit Problem for log-concave measures, and as pointed out in \cite{ABP}, implies that most marginals of log-concave measures are approximately Gaussian.

A stronger version of this conjecture was put forth by Bobkov and Koldobsky \cite{BobkovKoldobsky}. It may be equivalently formulated as stating that the ``thin-shell width" $\sqrt{\Var|X|}$ is bounded above by a universal constant $C$.

An even stronger conjecture is due to Kannan, Lov\'asz and Simonovits \cite{KLS}. In an equivalent form, it states that for any smooth function $f : \Real^n \rightarrow \Real$:
\[
\Var(f(X)) \leq C \E|\nabla f(X)|^2 ~.
\]
Applied to the function $f(x) = |x|^p$ with $p = c \sqrt{n}$, the KLS conjecture implies (see \cite{FleuryOnVarianceConjecture} and Section \ref{sec:deviation}) that:
\begin{equation} \label{eq:KLS-conc}
\P(\abs{|X| - \sqrt{n}} \geq t \sqrt{n}) \leq C \exp(-c \sqrt{n} t) \;\;\; \forall t \geq 0 ~.
\end{equation}

It was shown by G. Paouris \cite{Paouris-IsotropicTail} that the predicted positive deviation estimate (\ref{eq:KLS-conc}) indeed holds in the large:
\begin{equation} \label{eq:Paouris-large-dev-gen}
\P(|X| \geq (1+t) \sqrt{n}) \leq \exp(-c \sqrt{n} t) \;\;\; \forall t \geq C > 0 ~.
\end{equation}
Moreover, Paouris showed that when $A \in M_{n}(\Real)$ with $\norm{A}_{HS}^2 = n$, and $X$ is $\psi_\alpha$ ($\alpha \in [1,2]$) with constant $b_\alpha > 0$, then:
\begin{equation} \label{eq:Paouris-large-dev-alpha}
\P(|A X| \geq (1+t) \sqrt{n}) \leq \exp(-c (n/(b_\alpha^2 \norm{A}_{op}^2))^{\frac{\alpha}{2}} t) \;\;\; \forall t \geq C > 0 ~.
\end{equation}
Here $\norm{A}_{op}$ denotes the operator norm of $A$. Recall that $X$ (and its density) is said to be ``$\psi_\alpha$ with constant $b_\alpha$'' if:
\[
 \brac{\E \abs{\scalar{X,y}}^p}^{1/p} \leq b_\alpha p^{1/\alpha} \brac{\E \abs{\scalar{X,y}}^2}^{1/2} \;\;\; \forall p \geq 2 \;\;\; \forall y \in \Real^n ~.
\]
Note that this definition is linearly invariant and that necessarily $b_\alpha \geq 2^{-1/\alpha}$. We will simply say that ``$X$ is $\psi_\alpha$'', if it is $\psi_\alpha$ with a universal positive constant $C$.
By a result of Berwald \cite{BerwaldMomentComparison} or by Borell's Lemma \cite{Borell-logconcave} (see \cite[Appendix III]{Milman-Schechtman-Book}), it is well known that any $X$ with log-concave density is $\psi_1$ with $b_1 \leq C$, some universal constant, and so we only gain additional information when $\alpha > 1$.

Subsequently, it was shown by Paouris \cite{PaourisSmallBall} that under the same assumptions, the following small-ball estimate, analogous to the large deviation one (\ref{eq:Paouris-large-dev-alpha}), also holds:
\begin{equation} \label{eq:Paouris-small-ball}
\P(|A X| \leq \eps \sqrt{n}) \leq (C \eps)^{c (n/(b_\alpha^2 \norm{A}_{op}^2))^{\frac{\alpha}{2}}} \;\;\; \forall \eps \in (0,1/C) ~,
\end{equation}
for some constant $C > 1$.

The positive large-deviation estimate (\ref{eq:Paouris-large-dev-alpha}) is easily verified to be sharp (up to universal constants) for all $\alpha \in [1,2]$. The sharpness of (\ref{eq:Paouris-small-ball}) is not known, and in fact is intimately related to the Slicing Problem (see \cite{DafnisPaouris}). In any case, this leaves open the concentration estimates in the bulk: positive deviation $\P(|X| \geq (1+t) \sqrt{n})$ when $t \in [0,C]$, and negative deviation $\P(|X| \leq (1-t) \sqrt{n})$ when $t \in [0,c]$ ($c \in (0,1)$); in particular, this gives no information on the thin-shell $\sqrt{\Var{|X|}}$.

\medskip

In a breakthrough work, the first non-trivial estimate on the concentration of $|X|$ around its expectation was given by B. Klartag in \cite{KlartagCLP}, involving delicate logarithmic improvements in $n$ over the trivial bounds. This validated the conjectured thin-shell concentration (\ref{eq:ABP-conc}), allowing Klartag to resolve the Central Limit Problem for log-concave measures. A different proof continuing Paouris' approach was given by Fleury, Gu\'edon and Paouris in \cite{FleuryGuedonPaourisCLP}. Klartag then improved in \cite{KlartagCLPpolynomial} his estimates from logarithmic to polynomial in $n$ as follows (for any small $\eps > 0$):
\begin{equation} \label{eq:Klartag-deviation}
\P(\abs{|X| - \sqrt{n}} \geq t \sqrt{n}) \leq C_\eps \exp(-c_\eps n^{\frac{1}{3}-\eps} t^{\frac{10}{3}-\eps}) \;\;\; \forall t \in [0,1] ~.
\end{equation}
This implies in particular a thin-shell estimate of:
\[
\sqrt{\Var{|X|}} \leq C_\eps n^{\frac{1}{2} - \frac{1}{10} + \eps} ~.
\]
Note, however, that when $t=1/2$, (\ref{eq:Klartag-deviation}) does not recover the sharp positive large-deviation estimate of Paouris (\ref{eq:Paouris-large-dev-gen}).

Recently in \cite{FleuryImprovedThinShell}, B. Fleury improved Klartag's thin-shell estimate to:
\[
\sqrt{\Var{|X|}} \leq C n^{\frac{1}{2} - \frac{1}{8}} ~,
\]
by obtaining the following deviation estimates:
\[
\P(|X| \geq (1+t) \sqrt{n}) \leq C \exp(-c n^{\frac{1}{4}} t^2) \;\;\; \forall t \in [0,1] ~;
\]
\[
\P(|X| \leq (1-t) \sqrt{n}) \leq C \exp(-c n^{\frac{1}{8}} t) \;\;\; \forall t \in [0,1] ~.
\]
Note, however, that when $t=1/2$, Fleury's positive and negative large-deviation estimates are both inferior to those of Klartag, and so in the mesoscopic scale $t = n^{-\delta}$ ($\delta > 0$ small), Klartag's estimates still outperform Fleury's (and Paouris' ones are inapplicable). In addition, note that both Klartag and Fleury's estimates do not seem to improve under a $\psi_\alpha$ condition, contrary to the ones of Paouris. See also \cite{Bobkov-Nazarov,KlartagUnconditionalVariance,FleuryOnVarianceConjecture,EMilman-Gaussian-Marginals} for further related results.

\medskip

All of this suggests that one might hope for a concentration estimate which:
\begin{itemize}
\item Recovers Paouris' sharp positive large-deviation estimate (\ref{eq:Paouris-large-dev-alpha}).
\item Improves if $X$ is $\psi_\alpha$.
\item Improves the best-known thin-shell estimate of Fleury.
\item Improves the best-known mesoscopic-deviation estimate of Klartag.
\item Interpolates continuously between all scales of $t$ (bulk, mesoscopic, large-deviation).
\end{itemize}
The aim of this work is to provide precisely such an estimate.

\subsection{The Results}

Following Paouris, we formulate our main results in greater generality, allowing an application of a linear transformation to $X$.

\begin{thm} \label{thm:intro-deviation}
Let $X$ denote an isotropic random vector in $\Real^n$ with log-concave density, which is in addition $\psi_\alpha$ ($\alpha \in [1,2]$) with constant $b_\alpha$, and let $A \in M_{n}(\Real)$ satisfy $\norm{A}_{HS}^2 = n$. Then:
\begin{equation} \label{eq:deviation}
\P(\abs{|A X| - \sqrt{n}} \geq t \sqrt{n}) \leq C \exp(-c \n^{\frac{\alpha}{2}} \min(t^{2+\alpha},t)) \;\;\; \forall t \geq 0 ~,
\end{equation}
where:
\begin{equation} \label{eq:n0-def}
\n := \frac{n}{\norm{A}_{op}^2 b_\alpha^2} ~.
\end{equation}
In particular, we obtain the following thin-shell estimate:
\begin{equation} \label{eq:thin-shell}
\sqrt{\Var(|A X|)} \leq C n^{\frac{1}{2}} \n^{-\frac{\alpha}{2(2+\alpha)}} ~.
\end{equation}
\end{thm}

For concreteness and future reference, we state again the deviation estimates above and below the expectation separately: the constant $C$ in (\ref{eq:deviation}) may actually be removed in the former estimate:
\begin{equation} \label{eq:large-deviation}
\P(|A X| \geq (1+t) \sqrt{n}) \leq \exp(-c \n^{\frac{\alpha}{2}} \min(t^{2+\alpha},t)) \;\;\; \forall t \geq 0 ~;
\end{equation}
and combining our estimate (\ref{eq:deviation}) with Paouris' small-ball estimate (\ref{eq:Paouris-small-ball}), we obtain for the latter:
\begin{equation} \label{eq:small-ball}
\P(|A X| \leq (1-t) \sqrt{n}) \leq C \exp(-c \n^{\frac{\alpha}{2}} \max(t^{2+\alpha},\log\frac{c_2}{1-t}) ) \;\;\; \forall t \in [0,1] ~.
\end{equation}

Applying Theorem \ref{thm:intro-deviation} with $\alpha=1$ and $A = Id$, we obtain that for any isotropic $X$ with log-concave density, the above estimates hold with $\n \geq c n$, and in particular we deduce the following improved thin-shell estimate:
\begin{equation} \label{eq:improved-thin-shell}
\sqrt{\Var(|X|)} \leq C n^{\frac{1}{2} - \frac{1}{6}} ~.
\end{equation}
Also note that (\ref{eq:large-deviation}) recovers (up to constants) Paouris' sharp large-deviation estimate (\ref{eq:Paouris-large-dev-alpha}). Moreover, we obtain $\P(|A X| \geq (1+t) \sqrt{n}) \leq \exp(-C_t \n^{\frac{\alpha}{2}})$ and $\P(|A X| \leq \eps \sqrt{n}) \leq C' \exp(-C_\eps \n^{\frac{\alpha}{2}})$
for \emph{any} $t > 0$ and $\eps \in (0,1)$, whereas the estimates (\ref{eq:Paouris-large-dev-alpha}) and (\ref{eq:Paouris-small-ball}) only ensure that this holds for $t \geq C$ and $\eps \in (0,1/C)$, for some \emph{large enough} $C > 0$. It is also possible to recover Paouris' small-ball estimate (\ref{eq:Paouris-small-ball}), but this seems to require additional justification, which we leave for another note.

\medskip

Theorem \ref{thm:intro-deviation} is a standard consequence of (and essentially equivalent to) the following moment estimates, which are the main result of this work:

\begin{thm} \label{thm:intro-moments}
With the same assumptions and notation as in Theorem \ref{thm:intro-deviation},
for any $1 \leq \abs{p-2} \leq c_1 \n^{\frac{\alpha}{2(\alpha+2)}}$:
\begin{equation} \label{eq:small-moments}
1 - C \frac{\abs{p-2}}{\n^{\frac{\alpha}{\alpha+2}}} \leq \frac{\brac{\E|A X|^p}^{\frac{1}{p}}}{\brac{\E|A X|^2}^{\frac{1}{2}}} \leq 1 + C \frac{\abs{p-2}}{\n^{\frac{\alpha}{\alpha+2}}} ~,
\end{equation}
and for any $c_1 \n^{\frac{\alpha}{2(\alpha+2)}} \leq \abs{p-2} \leq c_2 \n^{\frac{\alpha}{2}}$:
\begin{equation} \label{eq:moments}
1 - C \brac{\frac{\abs{p-2}}{\n^{\frac{\alpha}{2}}}}^{\frac{1}{\alpha+1}} \leq \frac{\brac{\E|A X|^p}^{\frac{1}{p}}}{\brac{\E|A X|^2}^{\frac{1}{2}}} \leq 1 + C \brac{\frac{\abs{p-2}}{\n^{\frac{\alpha}{2}}}}^{\frac{1}{\alpha+1}} ~.
\end{equation}
\end{thm}

More precisely, we first derive a refined version of Theorem \ref{thm:intro-moments} with $A X$ replaced by $Y = (A X + G_n)/\sqrt{2}$, where $G_n$ denotes an independent standard Gaussian random vector in $\Real^n$. From this version, we derive the deviation estimates (\ref{eq:large-deviation}) and (\ref{eq:small-ball}) for $Y$ directly. Theorem \ref{thm:intro-deviation} for $A X$ then easily follows, but to deduce back the \emph{negative} moment estimates in (\ref{eq:moments}) for $A X$ up to $-p = c_2 \n^{\frac{\alpha}{2}}$ (or equivalently, the \emph{negative} deviation estimate (\ref{eq:small-ball})), we elude to the small-ball estimate (\ref{eq:Paouris-small-ball}). We remark that the lower bound $\abs{p-2} \geq 1$ in Theorem \ref{thm:intro-moments} may be replaced by any positive constant, leading to a different constant $C>0$ in the conclusion, and that as usual, the $L_0$-norm is interpreted as $\exp(\E \log |A X|)$.

\begin{rem}
Our choice to present the results assuming that $\norm{A}_{HS}^2 = n$ is purely for aesthetic reasons, facilitating the comparison to the previously known results.
Indeed, we can obviously remove this assumption by scaling $X$, and state all of our deviation estimates around (and relative to) the expected value $(\E |AX|^2)^{1/2} = \norm{A}_{HS}$ instead of $\sqrt{n}$. This leads to the following scale-invariant definition of $\n$ as $\n := \norm{A}_{HS}^2/(b_\alpha^2 \norm{A}_{op}^2)$, which naturally also appears in the work of Paouris \cite{Paouris-IsotropicTail,PaourisSmallBall}.
\end{rem}

Let us finally mention that by a standard application of a remarkable theorem due to Bobkov \cite{BobkovVarianceBound}, we improve the best-known general bound on the Cheeger constant $D_{Che}(\mu)$ of a probability measure $\mu$ in $\Real^n$ with isotropic log-concave density (we refer to \cite{BobkovVarianceBound,EMilman-RoleOfConvexity} for missing definitions and background). Bobkov's theorem states that for such measures $D_{Che}(\mu)^2 \geq c / (\E |X| \sqrt{\Var |X|})$ (where $X$ is distributed according to $\mu$), and so our improved thin-shell estimate (\ref{eq:improved-thin-shell}) implies:
\begin{cor}
Let $\mu$ denote a probability measure in $\Real^n$ with isotropic log-concave density. Then $D_{Che}(\mu) \geq c n^{-\frac{5}{12}}$.
\end{cor}
\noindent
This should be compared to the bound $D_{Che}(\mu) \geq c > 0$ conjectured by Kannan, Lov\'asz and Simonovits \cite{KLS}.
Note that our estimate improves all the way to $D_{Che}(\mu) \geq c n^{-\frac{3}{8}}$ when the density of $\mu$ is $\psi_2$.

\subsection{The Approach}

We assume throughout all proofs in this work that $\n$, and hence $n$, are greater than some large enough positive constant, since otherwise all stated results follow trivially (or easily, by inspecting the proof). Let $G_{n,k}$ denote the Grassmann manifold of all $k$-dimensional linear subspaces of $\Real^n$, and $SO(n)$ the group of rotations. Fixing a Euclidean structure on $\Real^n$, and given a linear subspace $F$, we denote by $S(F)$ and $B_2(F)$ the unit-sphere and unit-ball in $F$, respectively. When $F = \Real^n$, we simply write $S^{n-1}$ and $B_2^n$. We denote by $P_F$ the orthogonal projection onto $F$ in $\Real^n$, and given a random vector $Y$ with density $g$, we denote by $\pi_F g$ the marginal density of $g$ on $F$, i.e. the density of $P_F Y$. When $F = \textrm{span}(\theta)$, $\theta \in S^{n-1}$, we denote by $\pi_\theta g$ the density on $\Real$ given by $\pi_\theta g(t) := \pi_F g(t \theta)$.

For the proof of Theorem \ref{thm:intro-moments}, we use many of the ingredients developed previously by Klartag \cite{KlartagCLPpolynomial}, and adapted to the language of moments by Fleury \cite{FleuryOnVarianceConjecture,FleuryImprovedThinShell}:
\begin{itemize}
\item It is (almost) enough to verify (\ref{eq:small-moments}) and (\ref{eq:moments}) with $AX$ replaced by $Y = (A X + G_n)/\sqrt{2}$.
\item It is useful to first project $Y$ onto a lower-dimensional subspace $F \in G_{n,k}$. This idea also appears in essence in the work of Paouris \cite{Paouris-IsotropicTail}. Klartag and Paouris use V. Milman's approach to Dvoretzky's theorem \cite{VMilman-DvoretzkyTheorem,Milman-Schechtman-Book} for identifying lower-dimensional structures in most marginals $P_F Y$. Fleury, on the other hand, takes an average over the Haar measure on $G_{n,k}$, which is more efficient (see \cite{FleuryImprovedThinShell} or below):
\begin{equation} \label{eq:Gaussians}
\frac{\brac{\E|Y|^p}^{1/p}}{\brac{\E|Y|^2}^{1/2}} \leq \frac{\brac{\E_{F,Y} |P_F Y|^p}^{1/p}}{\brac{\E_{F,Y} |P_F Y|^2}^{1/2}} ~.
\end{equation}
\item Rewriting using the invariance of the Haar measure and polar coordinates:
\begin{equation} \label{eq:intro-ratio}
\frac{\brac{\E_{F,Y} |P_F Y|^p}^{1/p}}{\brac{\E_{F,Y} |P_F Y|^2}^{1/2}} = \frac{\brac{\E_{U} h_{k,p}(U)}^{1/p}}{\brac{ \E_{U} h_{k,2}(U)}^{1/2}} ~,
\end{equation}
where $U$ is uniformly distributed over $SO(n)$, $E_0 \in G_{n,k}$, $\theta_0 \in S(E_0)$, $g$ denotes the density of $Y$ in $\Real^n$, and $h_{k,p} : SO(n) \rightarrow \Real_+$ is defined as:
\begin{equation} \label{eq:intro-h_p}
 h_{k,p}(u) := \vol(S^{k-1}) \int_{0}^\infty t^{p+k-1} \pi_{u(E_0)} g(t u(\theta_0)) dt ~.
\end{equation}
To control the ratio in (\ref{eq:intro-ratio}), a good bound on the log-Lipschitz constant $L_{k,p}$ of $h_{k,p}$ is required.
\end{itemize}

Our main technical result in this work is the following improvement over the log-Lipschitz bounds of Klartag from \cite{KlartagCLPpolynomial}:
\begin{thm} \label{thm:intro-Lip}
Under the same assumptions as in Theorem \ref{thm:intro-deviation}, if $p \geq -k+1$ then $L_{k,p} \leq C \norm{A}_{op} b_\alpha \max(k,p)^{1/\alpha + 1/2}$.
\end{thm}
\noindent
Contrary to Klartag's analytical approach for controlling the log-Lipschitz constant, ours is completely based on geometric convexity arguments, employing the convex bodies $K_{k+q}$ introduced by K. Ball in \cite{Ball-kdim-sections}, and a variation on the $L_q$-centroid bodies, which were introduced by E. Lutwak and G. Zhang in \cite{LutwakZhang-IntroduceLqCentroidBodies}.

\medskip

Fleury proceeds by employing three additional ingredients:
\begin{itemize}
\item As shown by Borell \cite{BorellLyapunov} (see also \cite{BarlowMarshallProschan}), for any log-concave density $w$ on $\Real_+$:
\begin{equation} \label{eq:Borell-concave}
 q \mapsto \log \frac{\int_0^\infty t^{q} w(t) dt}{\Gamma(q+1)} \text{ is concave on $\Real_+$} ~.
\end{equation}
Consequently, $p \mapsto \log (h_{k,p}(u) / \Gamma(k+p))$ is concave on $p \in [-k+1,\infty)$ for any fixed $u \in SO(n)$.
This ingredient was also used in \cite{FleuryGuedonPaourisCLP}.
\item
As follows e.g. from the work of Bakry and \'Emery \cite{BakryEmery} (see also \cite{Ledoux-Book}), for any Lipschitz function $f : SO(n) \rightarrow \Real_+$, the following log-Sobolev inequality is satisfied (see Sections \ref{sec:Lip} and \ref{sec:moments} for definitions):
\begin{equation} \label{eq:log-Sob}
\Ent_{U}(f) \leq \frac{c}{n} \E_{U} (|\nabla f|^2/f) ~.
\end{equation}
\item
The latter log-Sobolev inequality implies via the Herbst argument, that for any log-Lipschitz function $f : SO(n) \rightarrow \Real_+$ with log-Lipschitz constant bounded above by $L$, the following reverse H\"{o}lder inequality holds (see \cite[(15)]{FleuryImprovedThinShell}):
\begin{equation} \label{eq:reverse-Holder}
(\E_{U} f^q)^{\frac{1}{q}} \leq \exp\brac{C \frac{L^2}{n} (q-r)} (\E_{U} f^r)^{\frac{1}{r}} \;\;\; \forall q > r > 0 ~.
\end{equation}
\end{itemize}

We proceed by using these ingredients as our predecessors, but our proof corrects the slight inefficiency of Fleury's approach in the resulting large-deviation estimate (witnessed by the comparison to Klartag's estimate earlier). The improvement here comes from the fact that we take the derivative in $p$ of (\ref{eq:Gaussians}), and optimize on the dimension $k$ for each $p$ separately, as opposed to optimizing on a single $k$ directly in (\ref{eq:Gaussians}). However, this by itself would not yield the improvement in the thin-shell estimate - the latter is due to our improved log-Lipschitz estimate in Theorem \ref{thm:intro-Lip}. Only by combining this improved log-Lipschitz estimate with our variation on Fleury's method, are we able to recover the sharp large-deviation estimates of Paouris (\ref{eq:Paouris-large-dev-alpha}). Moreover, the negative moment estimates of (\ref{eq:small-moments}) and (\ref{eq:moments}) are also obtained
almost for free, at least with $A X$ replaced by $Y$, after some slight additional justification for handling the $p$ moments in the range $p \in [-c \n^{-1/2} , c \n^{-1/2}]$.

\medskip

The rest of this work is organized as follows. In Section \ref{sec:Lip} we prove a more general version of Theorem \ref{thm:intro-Lip}. In Section \ref{sec:moments} we provide a complete proof of a refined version of Theorem \ref{thm:intro-moments}, with $A X$ replaced by $Y$, without eluding to (\ref{eq:Paouris-small-ball}). In Section \ref{sec:deviation}, we derive for completeness Theorem \ref{thm:intro-deviation} from Theorem \ref{thm:intro-moments}, and obtain the reduction from $A X$ to $Y$.
In the Appendix, we provide a proof of Proposition \ref{prop:Z-K-main} and other lemmas, whose purpose is to handle the case when $X$ is not centrally-symmetric (has non-even density).

\medskip
\noindent \textbf{Acknowledgement.} We thank Bo'az Klartag for his interest and comments and Matthieu Fradelizi for discussions. We also thank the anonymous referees for helpful suggestions. This work was done in part when the authors attended the Thematic Program on Asymptotic Geometric Analysis at the Fields Institute in Toronto.

\section{An improved log-Lipschitz estimate} \label{sec:Lip}

Let $M_{k,l}(\Real)$ denote the set of $k$ by $l$ matrices over $\Real$, and set $M_n(\Real) = M_{n,n}(\Real)$.
We equip
\[
SO(n) = \set{u \in M_{n}(\Real) ; u^t u = Id \;,\; det(u) = 1}
\]
with its standard (left and right) invariant Riemannian metric $g$, which we specify for concreteness on $T_{Id} SO(n)$, the tangent space at the identity element $Id \in SO(n)$. Fixing an orthonormal basis of $\Real^n$ and taking the derivative of the relation $u^t u = Id$, we see that this tangent space may be identified with all anti-symmetric matrices $\set{ B \in M_{n}(\Real) ; B^t + B = 0}$. Given $B \in T_{Id} SO(n)$, we set $|B|^2 = \scalar{B,B} := g_{Id}(B,B) = \frac{1}{2}\norm{B}_{HS}^2$, where recall the Hilbert-Schmidt norm of $A \in M_{k,l}(\Real)$ is given by $\norm{A}^2_{HS} := tr(A^t A) = \sum_{1 \leq i \leq k , 1 \leq j \leq l} A_{i,j}^2$. The factor of $\frac{1}{2}$ above is simply a convenience to ensure that
a full $2 \pi$ degree rotation in any two-plane leaving the orthogonal complement in place, has geodesic length $2 \pi$, and to prevent further appearances of factors like $\sqrt{2}$ later on. Up to this factor, this metric coincides with the one induced from the natural embedding $SO(n) \subset M_n(\Real)$ when $M_n(\Real)$ is equipped with the Hilbert-Schmidt metric (i.e. identified with the canonical Euclidean space on its $n^2$ entries).

\subsection{Main Result}

Throughout this section, let $Y$ denote a random vector in $\Real^n$ with log-concave density $g$ and barycenter at the origin.
Given an integer $k$ between $1$ and $n$, a real number $p \geq -k+1$, a linear subspace $E_0 \in G_{n,k}$ and $\theta_0 \in S(E_0)$, we recall the definition of the function $h_{k,p} : SO(n) \rightarrow \Real_+$:
\begin{equation} \label{eq:h_p}
 h_{k,p}(u) := \vol(S^{k-1}) \int_{0}^\infty t^{p+k-1} \pi_{u(E_0)} g(t u(\theta_0)) dt \;\;\; , \;\;\; u \in SO(n) ~.
\end{equation}
Note that $\pi_E g$ is log-concave for any $E \in G_{n,k}$ by the Pr\'ekopa--Leindler Theorem (e.g. \cite{GardnerSurveyInBAMS}).

When $Y = (X + G_n) / \sqrt{2}$, where (as throughout this work) $X$ denotes an isotropic random vector in $\Real^n$ with log-concave density, an upper bound on the log-Lipschitz constant (i.e. the Lipschitz constant of the logarithm) of:
\[
SO(n) \ni u \mapsto \pi_{u(E_0)} g(t u(\theta_0))
\]
was obtained by Klartag \cite[Lemma 3.1]{KlartagCLPpolynomial}, playing a crucial role in his polynomial estimates on the thin-shell of an isotropic log-concave measure. When $t \leq C \sqrt{k}$, Klartag's estimate is of the order of $k^2$. In \cite{FleuryImprovedThinShell}, Fleury defined a truncated version of (\ref{eq:h_p}), where the integral ranges up to $C \sqrt{k}$. Klartag's estimate obviously implies the same bound on the log-Lipschitz constant of this truncated version of $h_{k,p}$.

Our main technical result in this work is the following improved estimate on the log-Lipschitz constant of $h_{k,p}$, which is completely based on geometric convexity arguments. Note that we do not need any truncation, nor do we need to assume that $Y$ has been convolved with a Gaussian to obtain a meaningful estimate. However, the improvement over Klartag's $k^2$ bound appears after this convolution.

\begin{thm} \label{thm:Lip}
The log-Lipschitz constant $L_{k,p}$ of $h_{k,p}(u): SO(n) \rightarrow \Real_+$ is bounded above by $C \max(k,p) \dist(Z^+_{\max(k,p)}(g),B_2^n)$.
\end{thm}

Here $Z^+_q(w) \subset \Real^n$ ($q \geq 1$) denotes the \emph{one-sided} $L_q$-centroid body of the density $w$ (which may not have total mass one), defined via its support functional:
\[
 h_{Z^+_q(w)}(y) = \brac{2 \int_{\Real^n} \scalar{x,y}_+^q w(x) dx}^{1/q} ~,
\]
(here as usual $a_+ := \max(a,0)$). A dual variant of this definition (when $q \in (0,1)$) was also used by C. Haberl in \cite{HaberlLpIntersectionBodies}. When $w$ is even, this coincides with the more standard definition of the $L_q$-centroid body, introduced by E. Lutwak and G. Zhang in \cite{LutwakZhang-IntroduceLqCentroidBodies} (under a different normalization):
\[
h_{Z_q(w)}(y) = \brac{\int_{\Real^n} \abs{\scalar{x,y}}^q w(x) dx}^{1/q} ~.
\]
Clearly:
\[
Z^+_q(w) \subset 2^{1/q} Z_q(w) ~.
\]
In any case, when $w$ is the characteristic function of a set $K$, we denote $Z^+_q(K) := Z^+_q(1_K)$, and similarly for $Z_q(K)$. Lastly, the \emph{geometric distance} $\dist(K,L)$ between two subsets $K,L \subset \Real^n$ is defined as:
\[
\dist(K,L) := \inf \set{ C_2 / C_1 ; C_1 L \subset K \subset C_2 L ~,~ C_1,C_2 > 0 } ~.
\]

A very useful result for handling the non-even case is due to Gr\"{u}nbaum \cite{GrunbaumSymmetry}
(see also \cite[Formula (10)]{Fradelizi-Habilitation} or \cite[Lemma 3.3]{Bobkov-GaussianMarginals} for simplified proofs):

\begin{lem}[Gr\"{u}nbaum] \label{lem:barycenter}
Let $X_1$ denote a random variable on $\Real$ with log-concave density and barycenter at the origin. Then $\frac{1}{e} \leq \P(X_1 \geq 0) \leq 1-\frac{1}{e}$.
\end{lem}

Note that by definition, $Y$ (and its density $g$) is $\psi_\alpha$ ($\alpha \in [1,2]$) with constant $b_\alpha$ iff $Z_q(g) \subset b_\alpha q^{1/\alpha} Z_2(g)$ for all $q \geq 2$. Also recall that by a result of Berwald \cite{BerwaldMomentComparison} or as a consequence of Borell's Lemma \cite{Borell-logconcave} (see also \cite{Milman-Pajor-LK} or \cite[Appendix III]{Milman-Schechtman-Book}), a log-concave probability density $g$ is always $\psi_1$, and that moreover:
\begin{equation} \label{eq:Z_q-prop}
 1 \leq q_1 \leq q_2 \;\;\; \Rightarrow \;\;\; Z_{q_1}(g) \subset Z_{q_2}(g) \subset C \frac{q_2}{q_1} Z_{q_1}(g) ~.
\end{equation}
If in addition the barycenter of $g$ is at the origin, then repeating the argument leading to (\ref{eq:Z_q-prop}) and using Lemma \ref{lem:barycenter}, one verifies:
\begin{equation} \label{eq:Z^+_q-prop}
 1 \leq q_1 \leq q_2 \;\;\; \Rightarrow \;\;\; \brac{\frac{2}{e}}^{\frac{1}{q_1} - \frac{1}{q_2}} Z^+_{q_1}(g) \subset Z^+_{q_2}(g) \subset C \brac{\frac{2e-2}{e}}^{\frac{1}{q_1} - \frac{1}{q_2}} \frac{q_2}{q_1} Z^+_{q_1}(g) ~.
\end{equation}

When $g$ is isotropic, note that $Z_2(g) = B_2^n$, and one may similarly show (see Lemma \ref{lem:Z^+_2}) that $c B_2^n \subset Z^+_2(g) \subset \sqrt{2} B_2^n$.
It follows immediately from (\ref{eq:Z^+_q-prop}) that in that case $\dist(Z^+_k(g),B_2^n) \leq C k$, and we see that Theorem \ref{thm:Lip} recovers Klartag's $k^2$ order of magnitude when $p \leq k$ (which is the case of interest in the subsequent analysis).

\medskip

The improvement over Klartag's bound comes from the following elementary:
\begin{lem} \label{lem:add-G}
Let $X$ denote an isotropic random-vector in $\Real^n$ with log-concave density. Given $A \in M_{n}(\Real)$, set $Y = (AX + G_n)/\sqrt{2}$ and denote by $g$ its density. Then for all $q \geq 2$:
\begin{enumerate}
\item $Z_q^+(g) \supset c \sqrt{q} B_2^n$.
\item If $X$ is $\psi_\alpha$ ($\alpha \in [1,2]$) with constant $b_\alpha$, then $Z_q^+(g) \subset (C_1 \norm{A}_{op} b_\alpha q^{1/\alpha} + C_2 \sqrt{q}) B_2^n$.
\end{enumerate}
\end{lem}
\begin{proof}
Given $\theta \in S^{n-1}$, denote
$Y_1 = \pi_{\theta} Y$, $X_1 = \pi_{\theta} A X$ and $G_1 = \pi_{\theta} G_n$ (a one-dimensional standard Gaussian random variable). We have:
\[
 h^q_{Z^+_q(g)}(\theta) = 2 \E (Y_1)_+^q = \frac{2}{2^{q/2}} \E \brac{X_1 + G_1}_+^q \geq \frac{2}{2^{q/2}} \E (G_1)_+^q \P(X_1 \geq 0) ~.
\]
When $X$ is centrally-symmetric then $\P(X_1 \geq 0) = 1/2$. In the general case, since $X_1$ has log-concave density on $\Real$ and barycenter at the origin, Lemma \ref{lem:barycenter} implies that $\P(X_1 \geq 0) \geq 1/e$, and hence:
\[
 h^q_{Z^+_q(g)}(\theta) \geq \frac{1}{e 2^{q/2}} \E|G_1|^q ~,
\]
by the symmetry of $G_1$.
An elementary calculation shows that $c_1 \sqrt{q} \leq (\E |G_1|^q)^{1/q} \leq c_2 \sqrt{q}$ for all $q \geq 1$, concluding the proof of the first assertion. Similarly:
\[
\frac{1}{2} h^q_{Z_q^+(g)}(\theta) \leq h^q_{Z_q(g)}(\theta) = \E |Y_1|^q = \E \abs{\frac{X_1 + G_1}{\sqrt{2}}}^q \leq \frac{2^{q-1}}{2^{q/2}} \E (|X_1|^q + |G_1|^q) ~.
\]
Assuming that $X$ is $\psi_\alpha$ with constant $b_\alpha$ and isotropic, it follows that:
\[
(\E|X_1|^q)^{1/q} \leq b_\alpha q^{1/\alpha} (\E|X_1|^2)^{1/2} \leq b_\alpha q^{1/\alpha} \norm{A}_{op} ~,
\]
and the second assertion readily follows.
\end{proof}

\begin{cor} \label{cor:Lip}
Let $X$ denote an isotropic random-vector in $\Real^n$ with log-concave density, which is in addition $\psi_\alpha$ ($\alpha \in [1,2]$) with constant $b_\alpha$. Let $A \in M_{n}(\Real)$, set $Y = (A X + G_n)/\sqrt{2}$ and denote by $g$ the density of $Y$. Then:
\[
\dist(Z^+_q(g),B_2^n) \leq C_1 (1 + \norm{A}_{op} b_\alpha q^{1/\alpha - 1/2}) ~.
\]
Consequently, when $\norm{A}_{op} \geq 1$, Theorem \ref{thm:Lip} implies that:
\[
L_{k,p} \leq C_2 \norm{A}_{op} b_\alpha \max(k,p)^{1/\alpha + 1/2} ~.
\]
\end{cor}

\subsection{Proof of Theorem \ref{thm:Lip}}

For convenience, we assume that $2 \leq k \leq n/2$, although it will be clear from the proof that this is immaterial. By the symmetry and transitivity of $SO(n)$, and since $E_0 \in G_{n,k}$ was arbitrary, it is enough to bound $|\nabla_{u_0} \log h_{k,p}|$ at $u_0 = Id$.
We complete $\theta_0$ to an orthonormal basis $\set{\theta_0,e^2,\ldots,e^k}$ of $E_0$, and take $\set{e^{k+1},\ldots,e^n}$ to be any completion to an orthonormal basis of $\Real^n$. In this basis, the anti-symmetric matrix $M := \nabla_{Id} \log h_{k,p} \in T_{Id} SO(n)$ looks as follows:
\begin{equation} \label{eq:M}
M = \brac{
\begin{BMAT}(rc){c:c}{c:c}
 \blockmatrix{0.4in}{0.4in}{$M_1$}  & \blockmatrix{0.8in}{0.4in}{$M_2$} \\
\blockmatrix{0.4in}{0.8in}{$-M_2^t$} &  \blockmatrix{0.8in}{0.8in}{$0$}
\end{BMAT}
} ~,~
M_1 =
\brac{
\begin{BMAT}(rc){c:c}{c:c}
\blockmatrix{0.25in}{0.25in}{\scriptsize{$0$}} & \blockmatrix{0.6in}{0.25in}{\scriptsize{$V_1$}} \\
\blockmatrix{0.25in}{0.6in}{\scriptsize{$-V_1^t$}} &  \blockmatrix{0.6in}{0.6in}{\scriptsize{$0$}}
\end{BMAT}
} ~,~
M_2 =
\brac{
\begin{BMAT}(rc){c}{c:c}
\blockmatrix{1.0in}{0.1in}{\scriptsize{$V_2$}} \\
\blockmatrix{1.0in}{0.4in}{$V_3$}
\end{BMAT}
} ~,
\end{equation}
where $M_1 \in M_{k,k}(\Real)$, $M_2 \in M_{k,n-k}(\Real)$, $V_1 \in M_{1,k-1}(\Real)$, $V_2 \in M_{1,n-k}(\Real)$ and $V_3 \in M_{k-1,n-k}(\Real)$. Indeed, the lower $n-k$ by $n-k$ block of $M$ is clearly $0$, since rotations in $E_0^\perp$, the orthogonal complement to $E_0$, leave $\pi_{u(E_0)} g$ and hence $h_{k,p}$ unaltered; and the lower $k-1$ by $k-1$ block of $M_1$ is $0$ since rotations which fix $\theta_0$ and act invariantly on $E_0$ preserve $h_{k,p}$ as well. Consequently $|\nabla_{Id} \log h_{k,p}|^2 = \norm{V_1}_{HS}^2 + \norm{V_2}_{HS}^2 + \norm{V_3}_{HS}^2$. We will analyze the contribution of these three terms separately.

Denote by $T_i$ ($i=1,2,3$) the subspace of $T_{Id} SO(n)$ having the form (\ref{eq:M}) with $V_{j} = 0$ for $j \neq i$. Given $B \in T_i$, we call the geodesic in $SO(n)$ emanating from $Id$ in the direction of $B$, i.e. $\Real \ni s \mapsto u_s := \exp_{Id}(s B) \in SO(n)$, a \emph{Type-$i$ movement}. By definition, $\frac{d}{ds} u_s|_{s=0} = B$, and hence $\frac{d}{ds}  \log h_{k,p}(u_s) |_{s=0} = \scalar{\nabla_{Id} \log h_{k,p}, B}$. Clearly $\norm{V_i}_{HS} = \sup_{ 0 \neq B \in T_i} \scalar{\nabla_{Id} \log h_{k,p},B} / |B|$, so our goal now is to obtain a uniform upper bound on the derivative of $\log h_{k,p}$ induced by a Type-$i$ movement.

\medskip

To this end, we recall the following crucial fact, due to K. Ball \cite[Theorem 5]{Ball-kdim-sections} in the even case, and verified to still hold in the general one by Klartag \cite[Theorem 2.2]{KlartagPerturbationsWithBoundedLK}:
\begin{thm*}\label{thm:Ball}
Let $w$ denote a log-concave function on $\Real^m$ with $0 < \int w < \infty$ and $w(0) > 0$. Given $q \geq 1$, set:
\[
\norm{x} = \norm{x}_{K_{q}(w)} := \brac{q \int_{0}^\infty t^{q-1} w(t x) dt}^{-\frac{1}{q}} ~,~ x \in \Real^m ~.
\]
Then for all $x,y \in \Real^m$, $0 \leq \norm{x} < \infty$, $\norm{x}= 0$ iff $x = 0$, $\norm{\lambda x} = \lambda \norm{x}$ for all $\lambda \geq 0$, and $\norm{x + y} \leq \norm{x} + \norm{y}$.
\end{thm*}
We will thus say that $\norm{\cdot}_{K_{q}(w)}$ defines a norm, even though it may fail to be even, and denote by $K_q(w) := \{x \in \Real^m ; \norm{x}_{K_q(w)} \leq 1 \}$ its associated convex compact unit-ball.
Note that the constant $q$ in front of the integral above is simply a convenient normalization for later use.
We also set $\norm{x}_{\hat{K}_q(w)} := \max(\norm{x}_{K_q(w)},\norm{-x}_{K_q(w)})$, having unit-ball $\hat{K}_q(w) = K_q(w) \cap -K_q(w)$.
Note that the triangle inequality implies that:
\begin{equation} \label{eq:triangle}
\abs{\norm{x}_{K_{q}(w)} - \norm{y}_{K_{q}(w)}} \leq \norm{x-y}_{\hat{K}_q(w)} ~.
\end{equation}
Finally, note that since $B_2^m$ is centrally-symmetric, then $C_1 B_2^m \subset K \cap -K \subset K \subset C_2 B_2^m$ iff $C_1 B_2^m \subset K \subset C_2 B_2^m$, and hence:
\begin{equation} \label{eq:dist}
\frac{\norm{x}_{\hat{K}_q(w)}}{\norm{y}_{K_{q}(w)}} \leq \dist(K_q(w),B_2^m) \frac{|x|}{|y|} ~.
\end{equation}

\subsubsection{Type-1 movement}

Let $B \in T_1$ with $|B|=1$ generate a Type-1 movement $\set{u_s}$, and denote $\xi_0 = \frac{d}{ds} u_s(\theta_0) |_{s=0} \in T_{\theta_0} S(\Real^n)$. Using henceforth the natural embedding $T_{\theta} S(\Real^n) \subset T_{\theta} \Real^n \simeq \Real^n$, a Type-1 movement ensures that $u_s$ is a rotation in the $\set{\theta_0,\xi_0}$ plane and that $\xi_0$ lies in the orthogonal complement to $\theta_0$ in $E_0$, so $u_s(E_0)= E_0$. Also note since $|B|=1$ that $|\xi_0|=1$. Recalling the definition of $h_{k,p}$, we conclude that for such a movement:
\[
h_{k,p}(u_s) =  \vol(S^{k-1}) \int_{0}^\infty  t^{p+k-1} \pi_{E_0} g(t u_s(\theta_0)) dt = c_{p,k} \norm{u_s(\theta_0)}^{-(k+p)}_{K_{k+p}(\pi_{E_0} g)} ~,
\]
where $c_{p,k} = \vol(S^{k-1})/(k+p)$ is totally immaterial. Consequently:
\[
 \abs{\scalar{\nabla_{Id} \log h_{k,p},B}} = \abs{\left . \frac{d}{ds}  \log h_{k,p}(u_s) \right |_{s=0} } =
(k+p) \abs{\left . \frac{d}{ds} \log \norm{u_s(\theta_0)}_{K_{k+p}(\pi_{E_0} g)} \right |_{s=0} } ~.
\]
Since $\abs{\frac{d}{ds} \norm{u_s(\theta_0)}_{K_{k+p}(\pi_{E_0} g)}} \leq \norm{\frac{d}{ds} u_s(\theta_0)}_{\hat{K}_{k+p}(\pi_{E_0} g)}$ by the triangle-inequality (\ref{eq:triangle}), we conclude using (\ref{eq:dist}) that:
\[
 \abs{\scalar{\nabla_{Id} \log h_{k,p},B}} \leq (k+p) \frac{\norm{\xi_0}_{\hat{K}_{k+p}(\pi_{E_0} g)}}{\norm{\theta_0}_{K_{k+p}(\pi_{E_0} g)}} \leq (k+p) \dist(K_{k+p}(\pi_{E_0} g)) , B_2(E_0)) ~.
\]

\subsubsection{Type-2 movement}

Let $B \in T_2$ with $|B|=1$ generate a Type-2 movement $\set{u_s}$, and denote $\theta_s := u_s(\theta_0)$ and $\xi_s := \frac{d}{ds} \theta_s \in T_{\theta_s} S(\Real^n)$. The Type-2 movement ensures that $\xi_0 \in E_0^\perp$ and that $u_s$ is a rotation in the $\set{\theta_0,\xi_0} = \set{\theta_s,\xi_s}$ plane, and $|B|=1$ ensures that $|\xi_0|=1$. Denoting $E^1$ the orthogonal complement to $\theta_0$ in $E_0$, it follows that $u_s$ rotates $E_0$ into $E_s := u_s(E_0) = E^1 \oplus \text{span}\{\theta_s\}$. Consequently, $u_s$ leaves $H := E_0 \oplus \text{span}\{\xi_0\} = E_s \oplus \text{span}\{\xi_s\} \in G_{n,k+1}$ invariant, and therefore:
\[
h_{k,p}(u_s) =  \vol(S^{k-1}) \int_{0}^\infty \int_{-\infty}^{\infty} t^{p+k-1} \pi_{H} g(t \theta_s + r \xi_s) dr dt ~.
\]
Performing the change of variables $r = v t$, which is valid except at the negligible point $t=0$, we obtain:
\[
 h_{k,p}(u_s) = \vol(S^{k-1}) \int_{0}^\infty \int_{-\infty}^{\infty} t^{p+k} \pi_{H} g(t (\theta_s + v \xi_s)) dv dt = c_{p,k}
\int_{-\infty}^{\infty} \norm{\theta_s + v \xi_s}^{-(k+p+1)}_{K_{k+p+1}(\pi_H g)} dv ~,
\]
where $c_{p,k} = \vol(S^{k-1})/(k+p+1)$. Using that $\frac{d}{ds} \xi_s = -\theta_s$ and the triangle inequality (\ref{eq:triangle}) and (\ref{eq:dist}) for $\norm{\cdot}_{K_{k+p+1}(\pi_H g)}$, we obtain:
\begin{eqnarray*}
\abs{\scalar{\nabla_{Id} \log h_{k,p},B}} &=& \abs{\left . \frac{d}{ds} \log h_{k,p}(u_s) \right |_{s=0} }
\leq (k+p+1) \sup_{v \in \Real} \frac{\norm{\xi_0 - v \theta_0}_{\hat{K}_{k+p+1}(\pi_{H} g)}}{\norm{\theta_0 + v \xi_0}_{K_{k+p+1}(\pi_{H} g)}} \\
&\leq&
 (k+p+1) \dist(K_{k+p+1}(\pi_H g),B_2(H)) \sup_{v \in \Real} \frac{|\xi_0 - v \theta_0|}{|\theta_0 + v \xi_0|} \\
&=&
 (k+p+1) \dist(K_{k+p+1}(\pi_H g),B_2(H)) ~,
\end{eqnarray*}
where we have used the fact that $\theta_0$ and $\xi_0$ are orthogonal unit vectors in the last equality.

\subsubsection{Type-3 movement}

Finally, we analyze the most important movement type, which is responsible for a subspace of movements of dimension $(k-1)(n-k)$ (out of the $\text{dim } G_{n,k} +\text{dim } S^{k-1} = k (n-k) + (k-1)$ dimensional subspace of non-trivial movements).

Let $0 \neq B \in T_3$ generate a Type-3 movement $\set{u_s}$, and set  $e^j_s := u_s(e^j)$ and $f^j :=  \frac{d}{ds} e^j_s |_{s=0}$, $j=2,\ldots,k$. The Type-3 movement ensures that $u_s(\theta_0) = \theta_0$ and that all $f^j \in E_0^\perp$. Denote $F_0 := \text{span}\{f^2,\ldots,f^k\}$, and note that by slightly perturbing $B$ if necessary, we may assume that $F_0$ is $k-1$ dimensional. Finally, set $H = E_0 \oplus F_0 \in G_{n,2k-1}$, and notice that $H$ is invariant under $u_s$ (since $u_s$ is an isometry acting as the identity on the orthogonal complement). Consequently, $H = E_s \oplus F_s$, where $E_s := u_s(E_0)$ and $F_s := u_s(F_0)$, and therefore:
\[
 h_{k,p}(u_s) =  \vol(S^{k-1}) \int_{0}^\infty \int_{F_s} t^{p+k-1} \pi_{H} g (t \theta_0 + y) dy dt ~.
\]
Using the change of variables $y = z t$, we obtain (with $c_{p,k} = \vol(S^{k-1})/(2k-1+p)$):
\[
 h_{k,p}(u_s) = \vol(S^{k-1}) \int_{0}^\infty \int_{F_s} t^{p+2k-2} \pi_{H} g (t (\theta_0 + z)) dz dt = c_{p,k} \int_{F_s} \norm{\theta_0 + z}^{-(2k-1+p)}_{K_{2k-1+p}(\pi_H g)} dz ~,
\]
which we rewrite, since $u_s$ is orthogonal, as:
\[
 h_{k,p}(u_s) = c_{p,k} \int_{F_0} \norm{\theta_0 + u_s(z)}^{-(2k-1+p)}_{K_{2k-1+p}(\pi_H g)} dz ~.
\]
As usual, the triangle inequality (\ref{eq:triangle}) for $\norm{\cdot}_{K_{2k-1+p}(\pi_H g)}$ implies that:
\[
\abs{\scalar{\nabla_{Id} \log h_{k,p},B}} = \abs{\left . \frac{d}{ds}  \log h_{k,p}(u_s) \right |_{s=0}}
\leq (2k-1+p) \sup_{z \in F_0} \frac{\norm{B z}_{\hat{K}_{2k-1+p}(\pi_{H} g)}}{\norm{\theta_0 + z}_{K_{2k-1+p}(\pi_{H} g)}} ~,
\]
and so by (\ref{eq:dist}):
\begin{eqnarray*}
 \frac{\abs{\scalar{\nabla_{Id} \log h_{k,p},B}}}{(2k-1+p) \dist(K_{2k-1+p}(\pi_H g),B_2(H)) } & \leq &
\sup_{z \in F_0} \frac{|Bz|}{|\theta_0 + z|} \leq \norm{B}_{op} \sup_{z \in F_0} \frac{|z|}{\sqrt{1 + |z|^2}} \\
& \leq & \frac{\norm{B}_{HS}}{\sqrt{2}} = |B| ~,
\end{eqnarray*}
where we have used that $\theta_0$ is perpendicular to $F_0$, and that $\norm{B}_{op} \leq \norm{B}_{HS} / \sqrt{2}$ for any anti-symmetric matrix $B$,
as may be easily verified by using the Cauchy--Schwarz inequality.

\subsection{Distance of $K_{m+p}$ to Euclidean ball}

To conclude the proof of Theorem \ref{thm:Lip}, it remains to control the geometric distance of $K_{m+p}(\pi_H g)$ to a Euclidean ball, for $H \in G_{n,m}$ with $m$ of the order of $k$. To this end, we compare $K_{m+p}(\pi_H g)$ to $Z_q(\pi_H g) = P_H Z_{q}(g)$ for a suitably chosen $q \geq 1$. Our motivation comes from the groundbreaking work of Paouris \cite{Paouris-IsotropicTail}, who noted that:
\[
Z_q(\pi_H g)  = Z_q(K_{m+q}(\pi_H g)) ~,
\]
and using the inclusion $Z_q(K) \subset conv(K \cup -K)$ for any set $K$ of volume $1$, obtained an upper bound on $\vol(Z_q(\pi_H g))$ by bounding above $\vol(K_{m+q}(\pi_H g))$, enabling Paouris to deduce important features of $P_H Z_{q}(g)$. In this work, on the other hand, we take the converse path, passing from $K_{m+q}$ bodies to $Z_q$ ones, and consequently need to introduce the $Z^+_q$ bodies to handle non-even densities. Moreover, we require bounds on $Z^+_q(K)$ both from above and from below, which turn out to be more laborious in the non-even case (when $K$ is not centrally-symmetric).

Since the distance to the Euclidean ball cannot increase under orthogonal projections, and since $c_1 Z^+_k(g) \subset c_2 Z^+_m(g) \subset c_3 Z^+_{2k-1}(g) \subset c_4 Z^+_{k}(g)$ when $k \leq m \leq 2k-1$ by (\ref{eq:Z^+_q-prop}), it remains to establish the following:

\begin{thm}\label{thm:dist}
Let $w$ denote a log-concave function on $\Real^m$ with $0 < \int w < \infty$ and barycenter at the origin. Then for any $p \geq -m+1$:
\[
\dist (K_{m+p}(w) , B_2^m) \leq C \max\brac{\frac{m}{m+p},1} \dist (Z^+_{\max(p,m)}(w),B_2^m) ~.
\]
\end{thm}

For the proof, we recall several useful properties of the bodies $K_q(w)$ and $Z^+_q(K)$.
First, it is known (see \cite{BarlowMarshallProschan,Ball-kdim-sections,Milman-Pajor-LK} for the even case and \cite[Lemmas 2.5,2.6]{KlartagPerturbationsWithBoundedLK} or \cite[Lemma 3.2 and (3.12)]{PaourisSmallBall} for the general one) that under the assumptions of Theorem \ref{thm:dist}:
\begin{equation} \label{eq:K_q-prop}
1 \leq q_1 \leq q_2 \;\;\; \Rightarrow \;\;\; e^{-m(\frac{1}{q_1}-\frac{1}{q_2})} \frac{K_{q_1}(w)}{w(0)^{1/q_1}} \subset \frac{K_{q_2}(w)}{w(0)^{1/q_2}} \subset \frac{\Gamma(q_2+1)^{1/q_2}}{\Gamma(q_1+1)^{1/q_1}} \frac{K_{q_1}(w)}{w(0)^{1/q_1}} ~.
\end{equation}
Second, integration in polar coordinates (cf. \cite{Paouris-IsotropicTail}) directly shows that:
\begin{equation} \label{eq:ZK-Z}
Z^+_q(K_{m+q}(w)) = Z^+_q(w) ~.
\end{equation}
Lastly, we require the following proposition, which is well-known in the even-case (e.g. \cite[Lemma 4.1]{Paouris-Small-Diameter}), but requires more work in the general one (note for instance that the barycenter of $K_{m+q}(w)$ below need not be at the origin); its proof is postponed to the Appendix.
\begin{prop} \label{prop:Z-K-main} For any $q \geq 1$:
\begin{equation} \label{eq:Z-K-main}
C_1 Z^+_q(K_{m+q}(w)) \subset \vol(K_{m+q}(w))^{1/q} K_{m+q}(w) \subset C_2 Z^+_q(K_{m+q}(w)) \brac{\frac{\Gamma(m+q+1)}{\Gamma(m)\Gamma(q+1)}}^{1/q}
\end{equation}
\end{prop}

\begin{proof}[Proof of Theorem \ref{thm:dist}]
When $p \geq 1$, observe that (\ref{eq:Z-K-main}), (\ref{eq:ZK-Z}) and Stirling's formula, imply that:
\[
\dist(K_{m+p}(w) , B_2^m) \leq C \frac{p+m}{p} \dist(Z^+_p(w),B_2^m) ~,
\]
and so when $p \geq m$ the asserted claim follows. Otherwise, using (\ref{eq:K_q-prop}), Stirling's formula, (\ref{eq:Z-K-main}) and (\ref{eq:ZK-Z}), we see that if $q \geq \max(p,1)$ then:
\[
\dist(K_{m+p}(w) , B_2^m) \leq C_1 \frac{m+q}{m+p} \dist(K_{m+q}(w) , B_2^m) \leq C_2 \frac{m+q}{m+p} \frac{m+q}{q} \dist(Z^+_q(w),B_2^m) ~.
\]
Setting $q = m$, the case $p<m$ is also settled.
\end{proof}

The proof of Theorem \ref{thm:Lip} is now complete.

\section{Moment Estimates} \label{sec:moments}

Our goal in this section is to prove:
\begin{thm} \label{thm:moments-Y}
Let $X$ denote an isotropic random vector in $\Real^n$ with log-concave density, which is in addition $\psi_\alpha$ ($\alpha \in [1,2]$) with constant $b_\alpha$. Let $A \in M_{n}(\Real)$ satisfy $\norm{A}_{HS}^2 = n$, and set $Y := (A X + G_n)/\sqrt{2}$, where $G_n$ is an independent standard Gaussian random vector in $\Real^n$.
Then for any $\abs{p} \leq c \n^{\alpha/2}$:
\begin{equation} \label{eq:moments-Y}
1 - C \brac{\frac{\abs{p-2}}{\n^{\frac{\alpha}{2}}}}^{\frac{1}{\alpha+1}} \leq \frac{\brac{\E|Y|^p}^{\frac{1}{p}}}{\brac{\E|Y|^2}^{\frac{1}{2}}} \leq 1 + C \brac{\frac{\abs{p-2}}{\n^{\frac{\alpha}{2}}}}^{\frac{1}{\alpha+1}} ~.
\end{equation}
where $\n$ was defined in (\ref{eq:n0-def}).
\end{thm}

Note that by the Pr\'ekopa--Leindler Theorem, $Y$ itself has log-concave density. We also remark that it is possible to improve the moment estimates in the range $1 \leq \abs{p-2} \leq c_1 \n^{\frac{\alpha}{2(\alpha+2)}}$ exactly as in Theorem \ref{thm:intro-moments}, but we do not insist on this here.

\subsection{Passing to $SO(n)$}

We start by repeating the argument of Fleury for passing from integration on $\Real^n$ to $SO(n)$. Let $0 \neq \abs{p} \leq \frac{n-1}{2}$, and let $k$ denote an integer
 between $2$ and $n$  to be determined later on, so that in addition $\abs{p} \leq \frac{k-1}{2}$. Since $|x|^p = a_{n,k,p} \E_F |P_F x|^p$, where $F$ is uniformly distributed on $G_{n,k}$ (according to its Haar probability measure), we have:
\[
\frac{\E |Y|^p}{\E |G_n|^p} = \frac{\E \E_F |P_F Y|^p}{\E \E_F |P_F G_n|^p} = \frac{\E \E_F |P_F Y|^p}{\E |G_k|^p} ~,
\]
where $G_i$ denotes a standard Gaussian random vector on $\Real^i$. A direct calculation
shows that:
\[
\E |G_i|^p = 2^{\frac{p}{2}} \frac{\Gamma((p+i)/2)}{\Gamma(i/2)} ~,
\]
and hence:
\begin{equation} \label{eq:so1}
\E |Y|^p = \frac{\Gamma((p+n)/2)\Gamma(k/2)}{\Gamma(n/2)\Gamma((p+k)/2)} \E \E_F |P_F Y|^p ~.
\end{equation}
Passing to polar coordinates on $F \in G_{n,k}$ and using the invariance of the Haar measures on $G_{n,k}$, $S(F)$ and $SO(n)$ under the action of $SO(n)$, we verify that:
\begin{equation} \label{eq:so2}
\E \E_F |P_F Y|^p = \E_{U} h_{k,p}(U) ~,
\end{equation}
where $U$ is uniformly distributed on $SO(n)$.

\subsection{Controlling the derivative}

We now deviate from Fleury's argument and proceed to estimate:
\begin{equation} \label{eq:main-deriv}
\frac{d}{dp} \log( (\E |Y|^p)^{\frac{1}{p}} ) = \frac{d}{dp} \log((\E_{U} h_{k,p}(U))^{\frac{1}{p}}) + \frac{d}{dp} \brac{\frac{1}{p}  \log \frac{\Gamma((p+n)/2)\Gamma(k/2)}{\Gamma(n/2)\Gamma((p+k)/2)}} ~.
\end{equation}
Given $u \in SO(n)$, we introduce the (non-probability) measure $\mu_{u}$ on $\Real_+$ having density $\vol(S^{k-1}) t^{k-1} \pi_{u(F_0)} g (t u(\theta_0))$, where $g$ is the density of $Y$ on $\Real^n$. We define the (probability) measure $\mu_{k,p} := \E_U \mu_U$ on $\Real_+$, and write:
\[
h_{k,p}(u) = \E_{\mu_u}(t^p) ~,~ \E_{U} h_{k,p}(U) = \E_U \E_{\mu_U}(t^p) = \E_{\mu_{k,p}}(t^p) ~.
\]
Here and in the sequel we use the following convention: given a measure space $(\Omega,\mu)$, which does not necessarily have total mass $1$, and a measurable $f : \Omega \rightarrow \Real_+$, we set:
\[
\E_\mu f = \E_\mu(f) = \int f d\mu ~~,~~ \Ent_\mu(f) = \E_\mu(f \log f) - \E_\mu(f) \log(\E_\mu(f)) ~.
\]
A useful fact, easily verified by direct calculation, is that:
\[
\frac{d}{dp} \log ((\E_\mu f^p)^{\frac{1}{p}}) = \frac{1}{p^2} \frac{\Ent_\mu(f^p)}{\E_\mu(f^p)} ~.
\]

We proceed with estimating (\ref{eq:main-deriv}). As explained:
\begin{equation} \label{eq:main-deriv2}
\frac{d}{dp} \log((\E_{U} h_{k,p}(U))^{\frac{1}{p}}) = \frac{1}{p^2} \frac{\Ent_{\mu_{k,p}}(t^p)}{\E_{\mu_{k,p}}(t^p)} = \frac{1}{p^2} \frac{\Ent_{\mu_{k,p}}(t^p)}{\E_{U} h_{k,p}(U)} ~.
\end{equation}
Our main idea here is to decompose the numerator as follows:
\begin{equation} \label{eq:decomp}
\Ent_{\mu_{k,p}}(t^p) = \E_{U} \Ent_{\mu_U}(t^p) + \Ent_U \E_{\mu_U}(t^p) = \E_{U} \Ent_{\mu_U}(t^p) + \Ent_U h_{k,p}(U) ~.
\end{equation}

The contribution of the second term in (\ref{eq:decomp}) is controlled using the log-Sobolev inequality (\ref{eq:log-Sob}):
\begin{equation} \label{eq:es1}
\frac{1}{p^2} \frac{\Ent_U h_{k,p}(U)}{\E_U h_{k,p}(U)} \leq \frac{c}{p^2 n} \frac{\E_U (|\nabla \log h_{k,p}|^2(U) h_{k,p}(U))}{\E_U h_{k,p}(U)} \leq \frac{c L_{k,p}^2}{p^2 n} ~,
\end{equation}
where recall $L_{k,p}$ denotes the log-Lipschitz constant of $u \mapsto h_{k,p}(u)$. To control the contribution of the first term in (\ref{eq:decomp}), we first write given $u \in SO(n)$:
\[
\frac{1}{p^2} \frac{\Ent_{\mu_u}(t^p)}{\E_{\mu_u}(t^p)} = \frac{d}{dp} \log((\E_{\mu_u} t^p)^{\frac{1}{p}}) = \frac{d}{dp} \frac{1}{p} \brac{\log \frac{h_{k,p}(u)}{\Gamma(k+p)} - \log \frac{h_{k,0}(u)}{\Gamma(k)} + \log \frac{\Gamma(k+p)}{\Gamma(k)} + \log h_{k,0}(u)} ~.
\]
By Borell's concavity result (\ref{eq:Borell-concave}), we realize that:
\[
\frac{d}{dp} \frac{1}{p} \brac{\log \frac{h_{k,p}(u)}{\Gamma(k+p)} - \log \frac{h_{k,0}(u)}{\Gamma(k)}} \leq 0 ~,
\]
and hence:
\[
\frac{1}{p^2} \frac{\Ent_{\mu_u}(t^p)}{\E_{\mu_u}(t^p)} \leq \frac{d}{dp} \brac{\frac{1}{p} \log \frac{\Gamma(k+p)}{\Gamma(k)}} - \frac{1}{p^2} \log h_{k,0}(u) ~.
\]
Plugging this estimate back into (\ref{eq:main-deriv2}) and (\ref{eq:decomp}), we obtain:
\begin{equation} \label{eq:es2}
\frac{1}{p^2} \frac{\E_{U} \Ent_{\mu_U}(t^p)}{\E_U \E_{\mu_U}(t^p)} \leq \frac{d}{dp} \brac{\frac{1}{p} \log \frac{\Gamma(k+p)}{\Gamma(k)}} + \frac{1}{p^2} \frac{\E_U \log(1/ h_{k,0}(U)) h_{k,p}(U)}{\E_U h_{k,p}(U)} ~.
\end{equation}
By using the Jensen and Cauchy--Schwarz inequalities, we bound the second term by:
\[
\frac{\E_U \log(1/ h_{k,0}(U)) h_{k,p}(U)}{\E_U h_{k,p}(U)} \leq \log \brac{\frac{\E_U \frac{h_{k,p}(U)}{h_{k,0}(U)}}{\E_U h_{k,p}(U)}} \leq \log\brac{\frac{(\E_U h_{k,p}(U)^2)^{1/2}}{\E_U h_{k,p}(U)} (\E_U h_{k,0}(U)^{-2})^{1/2}} ~.
\]
We now use the reverse H\"{o}lder inequality (\ref{eq:reverse-Holder}) for comparing the various moments above. Denoting $\norm{f}_{q} := (\E_U |f(U)|^q)^{1/q}$, we have:
\[
\norm{h_{k,p}}_2 \leq \exp\brac{\frac{C L_{k,p}^2}{n}} \norm{h_{k,p}}_1 ~,~
\]
\[
\norm{h_{k,0}^{-1}}_2 \leq \exp\brac{\frac{2 C L_{k,0}^2}{n}} \norm{h_{k,0}^{-1}}_0 =
\exp\brac{\frac{2 C L_{k,0}^2}{n}} \frac{1}{\norm{h_{k,0}}_0} \leq \exp\brac{\frac{3 C L_{k,0}^2}{n}} \frac{1}{\norm{h_{k,0}}_1} ~.
\]
Since $\norm{h_{k,0}}_1 = \E_U h_{k,0}(U) = \E_{\mu_{k,p}}(1) = 1$, we conclude that:
\begin{equation} \label{eq:es3}
\frac{1}{p^2} \frac{\E_U \log(1/ h_{k,0}(U)) h_{k,p}(U)}{\E_U h_{k,p}(U)} \leq \frac{C}{p^2 n} (L_{k,p}^2 + 3 L_{k,0}^2) ~.
\end{equation}

Now, plugging all the estimates (\ref{eq:es1}), (\ref{eq:es2}), (\ref{eq:es3}) into (\ref{eq:main-deriv2}) using the decomposition (\ref{eq:decomp}), and plugging the result into (\ref{eq:main-deriv}), we obtain:
\[
\frac{d}{dp} \log( (\E |Y|^p)^{\frac{1}{p}} ) \leq \frac{c}{p^2 n}( 2 L_{k,p}^2 + 3 L_{k,0}^2) + \frac{d}{dp} \brac{\frac{1}{p} \log \frac{\Gamma(k+p)}{\Gamma(k)}} + \frac{d}{dp} \brac{\frac{1}{p} \log \frac{\Gamma((p+n)/2)\Gamma(k/2)}{\Gamma(n/2)\Gamma((p+k)/2)}} ~.
\]

\subsection{Optimizing on the dimension}

As observed by Fleury in \cite{FleuryImprovedThinShell}, using that the function $\frac{d}{dp} \log \Gamma(p)$ is concave, one easily verifies that the last term above satisfies:
\begin{equation} \label{eq:Gamma-decr}
\frac{d}{dp} \brac{\frac{1}{p} \log \frac{\Gamma((p+n)/2)\Gamma(k/2)}{\Gamma(n/2)\Gamma((p+k)/2)}} \leq 0 ~.
\end{equation}
Since the contribution of this term is insignificant relative to the second one, we simply use (\ref{eq:Gamma-decr}) as an upper bound.
For the second term, for any $q \neq 0$ having the same sign as $p$ and such that $k+p+q > 0$, we estimate using Jensen's inequality:
\begin{eqnarray*}
\frac{d}{dp} \brac{\frac{1}{p} \log \frac{\Gamma(k+p)}{\Gamma(k)}} = \frac{1}{p q} \frac{\int_0^\infty \log(t^q) t^{p+k-1} \exp(-t) dt}{\Gamma(p+k)} - \frac{1}{p^2} \log \frac{\Gamma(k+p)}{\Gamma(k)} \\
\leq \frac{1}{p q} \log \frac{\Gamma(k+p+q)}{\Gamma(k+p)} - \frac{1}{p^2} \log \frac{\Gamma(k+p)}{\Gamma(k)} =
\frac{1}{p} \log \brac{\frac{\Gamma(k+p+q)^{1/q}}{\Gamma(k+p)^{1/q}} \frac{\Gamma(k)^{1/p}}{\Gamma(k+p)^{1/p}}} ~.
\end{eqnarray*}
Applying Stirling's formula, setting $q = (p+k-1) \frac{p}{k-1}$, which indeed satisfies the above restrictions since $p \geq -\frac{k-1}{2}$, and using the latter condition on $p$, one verifies that:
\begin{equation} \label{eq:Stirling-Bound}
\frac{d}{dp} \brac{\frac{1}{p} \log \frac{\Gamma(k+p)}{\Gamma(k)}} \leq \frac{C}{k} ~;
\end{equation}
see also Remark \ref{rem:weaker-concavity} below for an alternative derivation.
Plugging our estimates for $L_{k,q}$ obtained in Corollary \ref{cor:Lip}, and noting that $\norm{A}_{op} \geq 1$ since $\norm{A}_{HS}^2 = n$, we conclude that if $X$ is $\psi_\alpha$ ($\alpha \in [1,2]$) with constant $b_\alpha$, then:
\begin{equation} \label{eq:key}
\frac{d}{dp} \log( (\E |Y|^p)^{\frac{1}{p}} ) \leq C \brac{\frac{b_\alpha^2 \norm{A}_{op}^2 k^{1 + 2/\alpha}}{p^2 n} + \frac{1}{k}} = C \brac{\frac{k^{1 + 2/\alpha}}{p^2 \n} + \frac{1}{k}} ~,
\end{equation}
for all integers $k$ in $[\max(2,2\abs{p}+1),n]$.
Optimizing on $k$ in that range, we set:
\[
k = \lceil \abs{p}^{1/\beta} \n^{1/(2\beta)} \rceil  ~,~ \beta := 1 + \frac{1}{\alpha} ~,
\]
which is guaranteed to be in the desired range whenever $\abs{p} \in [4 \n^{-1/2} , \frac{1}{64} \n^{\alpha/2}]$, as may be easily verified using that $\norm{A}_{op} \geq 1$ and $b_\alpha \geq 2^{-1/\alpha}$. Consequently, for such $p$, we obtain:
\[
\frac{d}{dp} \log( (\E |Y|^p)^{\frac{1}{p}} ) \leq \frac{C_2}{\abs{p}^{1/\beta}  \n^{1/(2\beta)}} ~.
\]
Setting $p_0 := 4 \n^{-1/2}$, we may assume that $p_0 \leq 2$ since $\n$ was assumed in the Introduction to be large enough (otherwise the statement of Theorem \ref{thm:moments-Y} follows easily), and so integrating over $p$ and adjusting constants, we obtain:
\begin{equation} \label{eq:pos-moments}
\exp\brac{- C \brac{\frac{\abs{p-2}}{\n^{\frac{\alpha}{2}}}}^{\frac{1}{\alpha+1}}} \leq \frac{\brac{\E|Y|^p}^{\frac{1}{p}}}{\brac{\E|Y|^2}^{\frac{1}{2}}} \leq \exp\brac{C \brac{\frac{\abs{p-2}}{\n^{\frac{\alpha}{2}}}}^{\frac{1}{\alpha+1}}} \;\;\; \forall p \in [p_0 , \frac{1}{64} \n^{\alpha/2}] ~,
\end{equation}
and:
\begin{equation} \label{eq:neg-moments}
 \frac{\brac{\E|Y|^p}^{\frac{1}{p}}}{\brac{\E|Y|^{-p_0}}^{-\frac{1}{p_0}}} \geq \exp\brac{- C \brac{\frac{\abs{p-p_0}}{\n^{\frac{\alpha}{2}}}}^{\frac{1}{\alpha+1}}} \;\;\; \forall p \in [-\frac{1}{64} \n^{\alpha/2},-p_0] ~.
\end{equation}

\subsection{Moments near $0$}

It remains to bridge the gap between the $p_0$ and $-p_0$ moments. Note that since we assume that $p_0 \leq 2$ and that $n$ is larger than some constant, then $p_0 \leq \frac{k_0-1}{2}$ for e.g. $k_0 = 5$.
Unfortunately, in the range $p \in [-p_0,p_0]$, our key estimate (\ref{eq:key}) only yields (using $k=k_0$):
\begin{equation} \label{eq:wasteful}
\frac{d}{dp} \log( (\E |Y|^p)^{\frac{1}{p}} ) \leq \frac{C}{p^2 \n} ~,
\end{equation}
which in particular is not integrable at $0$. We consequently treat this gap differently, by reproducing Fleury's argument from \cite{FleuryImprovedThinShell}.

Note that by Borell's concavity result (\ref{eq:Borell-concave}), we have:
\[
h_{k_0,p_0}^{\frac{1}{2}}(u) h_{k_0,-p_0}^{\frac{1}{2}}(u) \leq \frac{\brac{\Gamma(k_0+p_0) \Gamma(k_0-p_0)}^{\frac{1}{2}}}{\Gamma(k_0)} h_{k_0,0}(u) \leq (1+C_2 p_0^2) h_{k_0,0}(u) ~.
\]
Taking expectation, denoting by $\Cov$ the covariance, and using the Cauchy--Schwarz inequality, we obtain:
\begin{eqnarray*}
(1+C_2 p_0^2) & \geq & \E_U h_{k_0,p_0}^{\frac{1}{2}}(U) \E_U h_{k_0,-p_0}^{\frac{1}{2}}(U) + \Cov_U (h_{k_0,p_0}^{\frac{1}{2}}(U),  h_{k_0,-p_0}^{\frac{1}{2}}(U)) \\
&\geq & \E_U h_{k_0,p_0}^{\frac{1}{2}}(U) \E_U h_{k_0,-p_0}^{\frac{1}{2}}(U) - \sqrt{\Var_U(h_{k_0,p_0}^{\frac{1}{2}}(U)) \Var_U(h_{k_0,-p_0}^{\frac{1}{2}}(U))} \\
& = & \E_U h_{k_0,p_0}^{\frac{1}{2}}(U) \E_U h_{k_0,-p_0}^{\frac{1}{2}}(U) \\
&  & - \brac{\E_U h_{k_0,p_0}(U) - (\E_U h_{k_0,p_0}^{\frac{1}{2}}(U))^2}^{\frac{1}{2}} \brac{\E_U h_{k_0,-p_0}(U) - (\E_U h_{k_0,-p_0}^{\frac{1}{2}}(U))^2}^{\frac{1}{2}} ~.
\end{eqnarray*}
Using the reverse H\"{o}lder inequality (\ref{eq:reverse-Holder}) for comparing the $L_{1/2}$ and $L_1$ norms of $h_{k_0,p_0}$ and $h_{k_0,-p_0}$, we obtain:
\[
(1+C_2 p_0^2) \geq \brac{\E_U h_{k_0,p_0}(U) \E_U h_{k_0,-p_0}(U)}^{1/2} \brac{\exp\brac{- \frac{C}{2} \frac{L_{k_0,p_0}^2 + L_{k_0,-p_0}^2}{n}} - C \frac{L_{k_0,p_0} L_{k_0,-p_0}}{n}} ~.
\]
By Corollary \ref{cor:Lip} we know that $L_{k_0,p_0} , L_{k_0,-p_0} \leq C_3 \norm{A}_{op} b_\alpha k_0^{1/\alpha + 1/2}$, and we conclude that:
\[
\frac{(\E_U h_{k_0,p_0}(U))^{\frac{1}{p_0}}}{(\E_U h_{k_0,-p_0}(U))^{-\frac{1}{p_0}}} \leq \brac{1 + \frac{C_4}{\n}}^{\frac{2}{p_0}} \leq 1+ \frac{C_5}{\sqrt{\n}} ~.
\]
Finally, using (\ref{eq:so1}), (\ref{eq:so2}) and (\ref{eq:Gamma-decr}), we see that:
\[
\frac{(\E |Y|^{p_0})^{\frac{1}{p_0}}}{(\E |Y|^{-p_0})^{-\frac{1}{p_0}}}  \leq \frac{(\E_U h_{k_0,p_0}(U))^{\frac{1}{p_0}}}{(\E_U h_{k_0,-p_0}(U))^{-\frac{1}{p_0}}}  \leq 1+ \frac{C_5}{\sqrt{\n}} ~.
\]
This fills the remaining gap, and together with (\ref{eq:pos-moments}) and (\ref{eq:neg-moments}), the assertion of Theorem \ref{thm:moments-Y} follows.

\begin{rem}
Examining the proof in the case $\alpha=1$ and $A=Id$, it is easy to verify that if the log-Lipschitz constant $L_{k,p}$ of $h_{k,p} : SO(n) \rightarrow \Real_+$ satisfies:
\[
2 \leq p \leq k \;\;\; \Rightarrow \;\;\; L_{k,p} \leq C p^\beta k^\gamma ~,~ \beta,\gamma \in \Real ~,
\]
then the sharp large-deviation estimate $\P(|X| \geq C \sqrt{n}) \leq \exp(-\sqrt{n})$ is recovered if and only if $\beta + \gamma = 3/2$. Of course, since $p \leq k$, it is better to have larger $\beta$, and this affects the resulting thin-shell estimate. Our estimates yield $\beta=0$ and $\gamma=3/2$. The wasteful bound (\ref{eq:wasteful}) when $p$ is close to $0$ perhaps suggests that we should expect to have $\beta=1$ and $\gamma=1/2$.
\end{rem}

\begin{rem} \label{rem:weaker-concavity}
It is possible to avoid the delicate calculation based on Stirling's formula leading to the bound (\ref{eq:Stirling-Bound}), by replacing Borell's concavity result (\ref{eq:Borell-concave}) in our derivation above by a slightly weaker concavity result due to Bobkov \cite{Bobkov-SpectralGapForSphericallySymmetric}. It states that for any log-concave density $w$ on $\Real_+$:
\[
 q \mapsto \log \frac{\int_0^\infty t^{q} w(t) dt}{q^q} \text{ is concave on $\Real_+$} ~.
\]
\end{rem}

\section{Deviation Estimates} \label{sec:deviation}

A completely standard consequence of Theorem \ref{thm:moments-Y} is the following:
\begin{thm} \label{thm:deviation-Y}
With the same assumptions and notation as in Theorem \ref{thm:moments-Y}:
\begin{equation} \label{eq:large-deviation-Y}
\P(|Y| \geq (1+t) \sqrt{n}) \leq \exp(-c \n^{\frac{\alpha}{2}} \min(t^{2+\alpha},t)) \;\;\; \forall t \geq 0 ~,
\end{equation}
and:
\begin{equation} \label{eq:small-ball-Y}
\P(|Y| \leq (1-t) \sqrt{n}) \leq C \exp(-c \n^{\frac{\alpha}{2}} \max(t^{2+\alpha},\log\frac{c_2}{1-t})) \;\;\; \forall t \in [0,1] ~.
\end{equation}
\end{thm}
For completeness, we provide a proof.

\begin{proof}
Set:
\[
\eps_{\n,\alpha} := \min\brac{1,\frac{2^{\frac{\alpha+2}{\alpha+1}} C}{\n^{\frac{\alpha}{2(\alpha+1)}}}} ~,
\]
and note that there exists a constant $t_0 \in (0,1]$, so that:
\begin{eqnarray}
\label{eq:p-t}
\forall t \in (\eps_{\n,\alpha} , t_0 ] & & \exists p_1 \in (4,c \n^{\alpha/2}] \;\;\; \text{such that} \;\;\;
t = 2 C \frac{(p_1-2)^{\frac{1}{\alpha+1}}}{\n^{\frac{\alpha}{2(\alpha+1)}}} ~, \\
\label{eq:p-t-neg}
& & \!\!\!\!\! \exists p_2 \in [-c \n^{\alpha/2},0) \;\;\; \text{such that} \;\;\;
t = 2 C \frac{\abs{p_2-2}^{\frac{1}{\alpha+1}}}{\n^{\frac{\alpha}{2(\alpha+1)}}} ~.
\end{eqnarray}
Here $c,C>0$ are the two constants appearing in Theorem \ref{thm:moments-Y}, which guarantee that:
\[
\brac{1 - \frac{t}{2}} \sqrt{n} \leq (\E|Y|^{p_2})^{\frac{1}{p_2}} \leq (\E|Y|^{p_1})^{\frac{1}{p_1}} \leq \brac{1 + \frac{t}{2}} \sqrt{n} ~.
\]

Since $\frac{1+t}{1+t/2} \geq 1+t/3$ for $t \in [0,1]$, we obtain by the Markov--Chebyshev inequality:
\[
\P(|Y| \geq (1+t) \sqrt{n}) \leq \P(|Y| \geq (1+t/3) (\E|Y|^{p_1})^{\frac{1}{p_1}}) \leq (1 + t/3)^{-p_1} \leq \exp(-p_1 t/4) ~.
\]
Expressing $p_1$ as a function of $t$ for $t$ in the range specified in (\ref{eq:p-t}), and plugging this above, we obtain:
\[
\P(|Y| \geq (1+t) \sqrt{n}) \leq \exp(-c_1 \n^{\alpha/2} t^{2+\alpha}) \;\;\;
\forall t \in [\eps_{\n,\alpha},t_0] ~.
\]
To extend this estimate to the entire interval $[0,t_0]$, note that:
\[
\P(|Y| \geq (1+t) \sqrt{n}) \leq (1+t)^{-2} \leq \exp(-t/2) \;\;\; \forall t \in [0,\eps_{\n,\alpha}] ~,
\]
and so adjusting the constants appearing above:
\[ \P(|Y| \geq (1+t) \sqrt{n}) \leq \exp(- c_2 \n^{\alpha/2} t^{2+\alpha}) \;\;\; \forall t \in [0,t_0] ~.
\] Finally, a standard application of Borell's lemma \cite{Borell-logconcave} (e.g. as in \cite{Paouris-IsotropicTail}), ensures that:
\[
\P(|Y| \geq (1+t) \sqrt{n}) \leq \exp(- c_3 \n^{\alpha/2} t) \;\;\; \forall t \geq t_0 ~,
\]
concluding the proof of the positive deviation estimate (\ref{eq:large-deviation-Y}).

Similarly:
\[
\P(|Y| \leq (1-t) \sqrt{n}) \leq \P(|Y| \leq (1-t/2) (\E|Y|^{p_2})^{\frac{1}{p_2}}) \leq (1 - t/2)^{-p_2} \leq \exp(p_2 t/2) ~.
\]
Expressing $p_2$ as a function of $t$ for $t$ in the range specified in (\ref{eq:p-t-neg}), and plugging this above, we obtain:
\[
\P(|Y| \leq (1-t) \sqrt{n}) \leq C_2 \exp(-c \n^{\alpha/2} t^{2+\alpha}) \;\;\;
\forall t \in [\eps_{\n,\alpha},t_0] ~.
\]
Adjusting the value of $C_2$ above, the estimate extends to the entire range $t \in [0,t_0]$. Lastly, setting $p_3 = -c_3 \n^{\frac{\alpha}{2}}$ so that:
\[
(\E|Y|^{p_3})^{\frac{1}{p_3}} \geq \frac{1}{2} \sqrt{n} ~,
\]
we obtain for all $\eps \in (0,1/2)$:
\[
\P(|Y| \leq \eps \sqrt{n}) \leq \P(|Y| \leq 2\eps (\E|Y|^{p_3})^{\frac{1}{p_3}}) \leq (2 \eps)^{-p_3} = \exp\brac{-c_3 \n^{\frac{\alpha}{2}} \log\brac{\frac{1}{2\eps}}} ~.
\]
Adjusting all constants, the negative deviation estimate (\ref{eq:small-ball-Y}) follows.
\end{proof}

To conclude the proof of Theorems \ref{thm:intro-deviation} and \ref{thm:intro-moments}, we estimate the deviation of $AX$ by that of $Y$ exactly like Klartag \cite{KlartagCLP}. Indeed, according to the argument described in the proof of \cite[Proposition 4.1]{KlartagCLP}, we have:
\[
\P(|A X| \geq (1+t) \sqrt{n}) \leq C \P\brac{\frac{\abs{A X+G_n}}{\sqrt{2}} \geq \sqrt{\frac{(1+t)^2 + 1}{2}} \sqrt{n}} ~,
\]
and:
\[
\P(|A X| \leq (1-t) \sqrt{n}) \leq C \P\brac{\frac{\abs{A X+G_n}}{\sqrt{2}} \leq \sqrt{\frac{(1-t)^2 + 1}{2}} \sqrt{n}} ~,
\]
for some universal constant $C>1$. The deviation estimate (\ref{eq:deviation}) of Theorem \ref{thm:intro-deviation} immediately follows from the corresponding estimates of Theorem \ref{thm:deviation-Y}. However, the more refined deviation estimates (\ref{eq:large-deviation}) and (\ref{eq:small-ball}) do not follow:
(\ref{eq:large-deviation}) only follows up to the unnecessary constant $C$ in front of the estimate:
\begin{equation} \label{eq:up-almost}
\P(|A X| \geq (1+t) \sqrt{n}) \leq C \exp(-c \n^{\frac{\alpha}{2}} \min(t^{2+\alpha},t)) \;\;\; \forall t \geq 0 ~,
\end{equation}
and (\ref{eq:small-ball}) follows without the decay to $0$ as $t \rightarrow 1$:
\begin{equation} \label{eq:down-almost}
\P(|A X| \leq (1-t) \sqrt{n}) \leq C \exp(-c \n^{\frac{\alpha}{2}} t^{2+\alpha}) \;\;\; \forall t \in [0,1] ~.
\end{equation}

To resolve these last issues, we proceed as follows. The unnecessary constant $C>1$ in (\ref{eq:up-almost}) is easily removed e.g. by repeating the argument of Fleury from \cite{FleuryImprovedThinShell}. Indeed, when $p \geq 1$, by the symmetry and independence of $G_n$, convexity of $t \mapsto t^p$ and the Cauchy--Schwarz inequality, we have:
\begin{eqnarray*}
 \E|Y|^{2p} &=& E\brac{\frac{|AX + G_n|^2}{2}}^p = \frac{1}{2} \E \brac{\brac{\frac{|AX + G_n|^2}{2}}^p + \brac{\frac{|AX - G_n|^2}{2}}^p} \\
 &\geq& \E\brac{\frac{|AX|^2 + |G_n|^2}{2}}^p \geq \E |AX|^p |G_n|^p = \E|AX|^p \E|G_n|^p \\
 &\geq& \E|AX|^p (\E|G_n|^2)^{p/2} = n^{p/2} \E|AX|^p ~.
\end{eqnarray*}
Since $\E |AX|^2 = \E |Y|^2 = \norm{A}_{HS}^2 = n$, we deduce:
\begin{equation} \label{eq:reduction}
\frac{(\E|AX|^p)^{\frac{1}{p}}}{(\E|A X|^2)^{\frac{1}{2}}} \leq \brac{\frac{(\E|Y|^{2p})^{\frac{1}{2p}}}{(\E|Y|^2)^{\frac{1}{2}}}}^{2} \;\;\; \forall p \geq 1 ~.
\end{equation}
Consequently, the $p$-moment estimates of Theorem \ref{thm:moments-Y} hold equally true (after adjusting constants) with $Y$ replaced by $AX$, when $p \geq 3$.
In particular, the $p$-moment estimates (\ref{eq:moments}) of Theorem \ref{thm:intro-moments} for $p \geq c_1 \n^{\frac{\alpha}{2(\alpha+2)}}$ are obtained.
Repeating the relevant parts in the proof of Theorem \ref{thm:deviation-Y}, the desired positive deviation estimate (\ref{eq:large-deviation}) follows. Finally,
applying \cite[Lemma 6]{FleuryOnVarianceConjecture} to the deviation estimates of Theorem \ref{thm:intro-deviation}, the positive $p$-moment estimates are improved in the range $1 \leq p-2 \leq c_1 \n^{\frac{\alpha}{2(\alpha+2)}}$, obtaining the right-hand side of (\ref{eq:small-moments}); see also below for a sketch of an alternative derivation. This takes care of the positive moment and deviation estimates.

\medskip

Reducing from $AX$ to $Y$ the small-ball estimate (\ref{eq:small-ball}), or equivalently, the negative moment estimates of (\ref{eq:moments}), seems more involved, and further arguments are needed. We choose to bypass these here by simply employing Paouris' small-ball estimate (\ref{eq:Paouris-small-ball}), which together with (\ref{eq:down-almost}) yields for some $c_3 \leq 1$ the desired:
\begin{equation} \label{eq:small-ball-again}
\P(|AX| \leq (1-t) \sqrt{n}) \leq C_2 \exp(-c_2 \n^{\frac{\alpha}{2}} \max(t^{2+\alpha},\log\frac{c_3}{1-t}) ) \;\;\; \forall t \in [0,1] ~.
\end{equation}
The negative moment estimates of (\ref{eq:small-moments}) and (\ref{eq:moments}) then follow by integrating (\ref{eq:small-ball-again}) by parts.
Since the computation is not entirely straightforward when $1 \leq \abs{p-2} \leq c_1 \n^{\frac{\alpha}{2(\alpha+2)}}$, we sketch the argument, which is based on Fleury's derivation in \cite[Lemma 6]{FleuryOnVarianceConjecture} of \emph{positive} moment estimates from deviation estimates.
However, Fleury's technique does not seem to generalize to \emph{negative} moments, and so we provide an alternative proof, which is equally applicable to both positive and negative moments.

Denote $Z = |AX| / \sqrt{n}$, and note that $1 = \E Z^2 = \int_0^\infty \P(Z^2 \geq t) dt$. We consequently have for $p > 0$:
\begin{eqnarray*}
&  & \E Z^{-2p} = p \int_0^\infty t^{-(p+1)} \P(Z^2 \leq t) dt \\
&= & p \int_0^1 \P(Z^2 \leq t) (t^{-(p+1)} - 1) dt + p \int_0^1 (1 - \P(Z^2 \geq t)) dt + p \int_1^\infty (1 - \P(Z^2 \geq t)) t^{-(p+1)} dt \\
& = & p \int_1^\infty t^{-(p+1)} dt + p \int_0^1 \P(Z^2 \leq t) (t^{-(p+1)} - 1) dt + p \int_1^\infty \P(Z^2 \geq t) ( 1 - t^{-(p+1)}) dt \\
& = & 1 + p \int_0^1 \P(Z^2 \leq 1-s) ((1-s)^{-(p+1)} - 1) ds + p \int_0^\infty \P(Z^2 \geq 1+s) (1 - (1+s)^{-(p+1)}) ds ~.
\end{eqnarray*}
Assuming for simplicity that $p \geq 2$, we use (\ref{eq:small-ball-again}) to bound the first integral above, evaluating separately the intervals $[1-c_3^2/2,1]$, $[0,1/p]$ and $[1/p,1-c_3^2/2]$, and (\ref{eq:large-deviation}) to bound the second integral, evaluating separately the intervals $[0,1/p]$ and $[1/p,\infty)$. Using the obvious estimates:
\[
\P(Z^2 \leq 1-s) \leq \P(Z \leq 1-s/2) ~,~ \P(Z^2 \geq 1+s) \leq \P(Z \geq 1+c \min(s,s^{1/2})) ~ ~ \forall s \geq 0 ~ ;
\]
\[
(1-s)^{-(p+1)} - 1 \leq \begin{cases}   C p s & s \in [0,1/p] \\ \exp(C p s) & s \in [1/p,1/2] \end{cases} ~ , ~
1 - (1+s)^{-(p+1)} \leq \begin{cases} (p+1) s & s \in [0,1/p] \\ 1 & s \in [1/p,\infty) \end{cases} ~,
\]
we obtain:
\begin{eqnarray}
\label{eq:big-bound}
& & \E Z^{-2p} \leq 1 + p C_2 \int_0^{c_3^2/2} (\eps/c_3^2)^{\frac{c_2}{2} \n^{\frac{\alpha}{2}}} \eps^{-(p+1)}  d\eps \\
\nonumber &+& p^2 C_3 \int_{0}^{1/p} s \exp(-c_4 \n^{\frac{\alpha}{2}} s^{2+\alpha}) ds + p C_4 \int_{1/p}^{1 - c_3^2/2} \exp(-c_4 \n^{\frac{\alpha}{2}} s^{2+\alpha} + c_5 p s) ds \\
\nonumber &+& p^2 C_5 \int_0^{1/p} s \exp(-c_6 \n^{\frac{\alpha}{2}} s^{2+\alpha}) ds + p C_6 \int_{1/p}^\infty \exp(-c_6 \n^{\frac{\alpha}{2}} \min(s^{2+\alpha},s^{\frac{1}{2}})) ds ~.
\end{eqnarray}
When $2 \leq p \leq c_1 \n^{\frac{\alpha}{2(\alpha+2)}}$, this implies using $(1+2p x)^{\frac{1}{2p}} \leq 1+x$:
\begin{eqnarray*}
& & \brac{\E Z^{-2p}}^{\frac{1}{2p}} \leq 1 + C_2 2^{-c_7 \n^{\frac{\alpha}{2}}} + p C_7 \int_{0}^{\infty} s \exp(-c_7 \n^{\frac{\alpha}{2}} s^{2+\alpha}) ds \\
&+& C_8 \int_{1/p}^{\infty} \exp(-c_4 \n^{\frac{\alpha}{2}} s^{2+\alpha} + c_5 p s) ds + C_9 \int_{1/p}^\infty \exp(-c_6 \n^{\frac{\alpha}{2}} \min(s^{2+\alpha},s^{\frac{1}{2}})) ds ~.
\end{eqnarray*}
In this range of values for $p$, $1/p \geq c (p / \n^{\frac{\alpha}{2}})^{\frac{1}{1+\alpha}}$, and hence the integrand in the term involving $C_8$ is monotone decreasing.
A standard computation then confirms that, in this range, both integrals involving $C_8$ and $C_9$ are dominated by the one involving $C_7$, yielding the negative moment estimates of (\ref{eq:small-moments}); a similar argument does the job in the positive moment range. When $c_1 \n^{\frac{\alpha}{2(\alpha+2)}} \leq p \leq c_2 \n^{\alpha/2}$, we similarly verify from (\ref{eq:big-bound}) that:
\[
\brac{\E Z^{-2p}}^{\frac{1}{2p}} \leq \brac{C_{10} \frac{p}{\n^{\frac{\alpha}{2(\alpha+2)}}} + p C_4 \int_{-\infty}^{\infty} \exp(-c_4 \n^{\frac{\alpha}{2}} |s|^{2+\alpha} + c_5 p s) ds}^{\frac{1}{2p}} ~.
\]
Bounding the second (dominant) term using the Laplace method, we obtain the negative moment estimates of (\ref{eq:moments}), thereby concluding the proof of Theorem \ref{thm:intro-moments}.

\section*{Appendix}
\renewcommand{\thesection}{A}
\setcounter{thm}{0}
\setcounter{equation}{0}  \setcounter{subsection}{0}

In the Appendix, we prove several properties of the bodies $Z_q^+(K)$ (for $q \geq 1$) which are needed for the results of Section \ref{sec:Lip}.

Our main goal is to establish Proposition \ref{prop:Z-K-main}. For the proof, we require several lemmas.
Given $\theta \in S^{m-1}$, we denote $H_\theta^+ := \set{x \in \Real^m ; \scalar{x,\theta} \geq 0}$.

\begin{lem} \label{lem:A1}
Let $K$ denote a convex body in $\Real^m$, and given $\theta \in S^{m-1}$, denote $f_\theta = \pi_\theta \mathbf{1}_K$. Then:
\[
\brac{\frac{f_\theta(0)}{\norm{f_\theta}_{\infty}}}^{1/q} \brac{\frac{\Gamma(m)\Gamma(q+1)}{\Gamma(m+q+1)}}^{1/q}  h_K(\theta) \leq \frac{h_{Z_q^+(K)}(\theta)}{(2\vol(K \cap H_\theta^+))^{1/q}} \leq h_K(\theta) ~.
\]
\end{lem}
\begin{proof}
The right inequality is straightforward from the definitions. The left inequality is derived by following the proof of \cite[Lemma 4.1]{Paouris-Small-Diameter}, which uses
the fact that the $1/(m-1)$ power of any one-dimensional marginal of $K$ is a concave function.
\end{proof}

To control the left-most term in Lemma \ref{lem:A1}, we have:
\begin{lem} \label{lem:A2}
Let $\mu = f(x) dx$ denote a log-concave probability measure on $\Real$. Then for any $\eps > 0$:
\[
\eps \leq \int_{0}^{\infty} f(x) dx \leq 1-\eps \;\;\; \Rightarrow f(0) \geq \eps \norm{f}_\infty ~.
\]
\end{lem}
\noindent
This is essentially folklore (see e.g. \cite[Lemma 1.1]{Fradelizi-Habilitation}), but we include a proof for completeness. We refer the interested reader e.g. to \cite{Fradelizi-Habilitation} for the study of functional inequalities in the case of non-symmetric log-concave measures.
\begin{proof}
Let $F(x) = \int_{-\infty}^x f(t) dt$ and $G(x) = 1 - F(x) = \int_x^\infty f(t) dt$.
By the Pr\'ekopa--Leindler Theorem, both $F$ and $G$ are log-concave. Equivalently, this means that both $f / F$ and $-f / G$ are non-increasing.
Consequently $f(x) \leq f(y) \max(F(x)/F(y) , G(x)/G(y))$ for all $x,y \in \Real$. Using $y=0$ and the assumption that $F(0),G(0) \geq \eps$, the conclusion immediately follows.
\end{proof}

This reduces our task to showing:
\begin{lem} \label{lem:A3}
If $w$ is a log-concave function on $\Real^m$ with barycenter at the origin, then:
\[
\forall \theta \in S^{m-1} \;\;\;  \brac{\frac{\vol(K_{m+q}(w) \cap H^+_\theta)}{\vol(K_{m+q}(w))}}^{1/q} \geq c > 0 ~.
\]
\end{lem}
\begin{proof}
Note that we may normalize and rescale so that $w(0) = 1$ and $\int_{\Real^m} w(x) dx = 1$. Using polar-coordinates, we have for any convex (in fact, star-shaped) body $K$ containing the origin:
\begin{equation} \label{eq:polar1}
\vol(K \cap H_\theta^+) = \frac{1}{m} \int_{S^{m-1} \cap H_\theta^+} \norm{\xi}_{K}^{-m} d\xi ~.
\end{equation}
Using (\ref{eq:K_q-prop}), we see that:
\[
\forall \xi \in S^{m-1}  \;\;\; e^{-\frac{mq}{m+q}} \norm{\xi}_{K_{m}(w)}^{-m} \leq \norm{\xi}_{K_{m+q}(w)}^{-m} \leq \frac{\Gamma(m+q+1)^{\frac{m}{m+q}}}{\Gamma(m+1)} \norm{\xi}_{K_{m}(w)}^{-m} ~.
\]
Plugging this into (\ref{eq:polar1}) and using Stirling's formula, we verify that:
\begin{equation} \label{eq:K-ratio}
\forall \theta \in S^{m-1}  \;\;\; e^{-q} \leq \frac{\vol(K_{m+q}(w) \cap H_\theta^+)}{\vol(K_{m}(w) \cap H_\theta^+)} \leq C^q ~.
\end{equation}

Using (\ref{eq:polar1}), the definition of $K_m(w)$ and polar-coordinates again, we see that $\vol(K_m(w) \cap H_\theta^+) = \int_{H_\theta^+} w(x) dx = \P(W_1 \geq 0)$, where $W_1$ is the random variable on $\Real$ having density $\pi_\theta w$. Since this density is log-concave by the Pr\'ekopa--Leindler Theorem, and since the barycenter of $W_1$ is at the origin, Lemma \ref{lem:barycenter} implies that:
\begin{equation} \label{eq:lem-bary-again}
\frac{\vol(K_{m}(w) \cap H_\theta^+)}{\vol(K_{m}(w))} \geq \frac{1}{e} ~.
\end{equation}
Now decomposing $\vol = \vol|_{H_\theta^+} + \vol|_{H_{-\theta}^+}$, (\ref{eq:K-ratio}) and (\ref{eq:lem-bary-again}) imply the assertion.
\end{proof}

\begin{proof}[Proof of Proposition \ref{prop:Z-K-main}]
Applying Lemma \ref{lem:A1} with $K = K_{m+q}(w)$ and using Lemma \ref{lem:A3}, we obtain for all $\theta \in S^{m-1}$:
\[
c \brac{\frac{f_\theta(0)}{\norm{f_\theta}_{\infty}}}^{1/q} \brac{\frac{\Gamma(m)\Gamma(q+1)}{\Gamma(m+q+1)}}^{1/q} \leq \vol(K_{m+q}(w))^{-1/q} \frac{h_{Z_q^+(K_{m+q}(w))}(\theta)}{h_{K_{m+q}(w)}(\theta)} \leq C ~.
\]
Lemma \ref{lem:A2} together with Lemma \ref{lem:A3} imply that:
\[
\forall \theta \in S^{m-1} \;\;\; \brac{\frac{f_\theta(0)}{\norm{f_\theta}_{\infty}}}^{1/q} \geq c' > 0 ~,
\]
and hence:
\[
c'' \brac{\frac{\Gamma(m)\Gamma(q+1)}{\Gamma(m+q+1)}}^{1/q} K_{m+q}(w) \subset \vol(K_{m+q}(w))^{-1/q} Z_q^+(K_{m+q}(w)) \subset C K_{m+q}(w) ~.
\]
Rearranging terms, the assertion of Proposition \ref{prop:Z-K-main} follows.
\end{proof}

Finally, we prove:
\begin{lem} \label{lem:Z^+_2}
If $g : \Real^m \rightarrow \Real_+$ is a log-concave isotropic density then $Z^+_2(g) \supset c B_2^m$.
\end{lem}
\begin{proof}
Given $\theta \in S^{n-1}$, denote $g_0 := \pi_{\theta} g$; as usual, it is an isotropic log-concave probability density on $\Real$.
Comparing moments using the left-hand side of (\ref{eq:K_q-prop}) with $m=1$, $q_1=1$ and $q_2 = 3$, we obtain:
\begin{equation} \label{eq:triv1}
3 \int_0^\infty t^2 g_0(t) dt \geq \frac{\brac{\int_0^\infty g_0(t) dt}^3}{e^2 g_0(0)^2} ~.
\end{equation}
Applying now the reverse comparison using the right-hand side of (\ref{eq:K_q-prop}) for both directions $\theta$ and $-\theta$, and summing the resulting estimates, we obtain:
\begin{equation} \label{eq:triv2}
3 = 3 \int_{-\infty}^\infty t^2 g_0(t) dt \leq \frac{\Gamma(4)}{g_0(0)^2} \brac{\brac{\int_0^\infty g_0(t) dt}^3 + \brac{\int_{-\infty}^0 g_0(t) dt}^3} ~.
\end{equation}
Since the barycenter of $g_0$ is at the origin, we know by Lemma \ref{lem:barycenter} that:
\[
\int_{-\infty}^0 g_0(t) dt \leq (e-1) \int_0^\infty g_0(t) dt ~,
\]
and so we conclude from (\ref{eq:triv2}) that:
\[
\frac{\brac{\int_0^\infty g_0(t) dt}^3}{g_0(0)^2} \geq \frac{3}{\Gamma(4)(1 + (e-1)^3)} ~.
\]
Together with (\ref{eq:triv1}), the assertion follows with e.g. $c = (3 e^2 (1 + (e-1)^3))^{-1/2}$.
\end{proof}

\bibliographystyle{plain}

\begin{thebibliography}{10}

\bibitem{ABP}
M.~Anttila, K.~Ball, and I.~Perissinaki.
\newblock The central limit problem for convex bodies.
\newblock {\em Trans. Amer. Math. Soc.}, 355(12):4723--4735, 2003.

\bibitem{BakryEmery}
D.~Bakry and M.~{\'E}mery.
\newblock Diffusions hypercontractives.
\newblock In {\em S\'eminaire de probabilit\'es, XIX, 1983/84}, volume 1123 of
  {\em Lecture Notes in Math.}, pages 177--206. Springer, Berlin, 1985.

\bibitem{Ball-kdim-sections}
K.~Ball.
\newblock Logarithmically concave functions and sections of convex sets in
  $\mathbb{R}^n$.
\newblock {\em Studia Math.}, 88(1):69--84, 1988.

\bibitem{BarlowMarshallProschan}
R.~E. Barlow, A.~W. Marshall, and F.~Proschan.
\newblock Properties of probability distributions with monotone hazard rate.
\newblock {\em Ann. Math. Statist.}, 34:375--389, 1963.

\bibitem{BerwaldMomentComparison}
L.~Berwald.
\newblock Verallgemeinerung eines {M}ittelwertsatzes von {J}. {F}avard f\"ur
  positive konkave {F}unktionen.
\newblock {\em Acta Math.}, 79:17--37, 1947.

\bibitem{BobkovVarianceBound}
S.~Bobkov.
\newblock On isoperimetric constants for log-concave probability distributions.
\newblock In {\em Geometric aspects of functional analysis, Israel Seminar
  2004-2005}, volume 1910 of {\em Lecture Notes in Math.}, pages 81--88.
  Springer, Berlin, 2007.

\bibitem{Bobkov-GaussianMarginals}
S.~G. Bobkov.
\newblock On concentration of distributions of random weighted sums.
\newblock {\em Ann. Probab.}, 31(1):195--215, 2003.

\bibitem{Bobkov-SpectralGapForSphericallySymmetric}
S.~G. Bobkov.
\newblock Spectral gap and concentration for some spherically symmetric
  probability measures.
\newblock In {\em Geometric aspects of functional analysis}, volume 1807 of
  {\em Lecture Notes in Math.}, pages 37--43. Springer, Berlin, 2003.

\bibitem{BobkovKoldobsky}
S.~G. Bobkov and A.~Koldobsky.
\newblock On the central limit property of convex bodies.
\newblock In {\em Geometric aspects of functional analysis}, volume 1807 of
  {\em Lecture Notes in Math.}, pages 44--52. Springer, Berlin, 2003.

\bibitem{Bobkov-Nazarov}
S.~G. Bobkov and F.~L. Nazarov.
\newblock On convex bodies and log-concave probability measures with
  unconditional basis.
\newblock In {\em Geometric Aspects of Functional Analysis}, volume 1807 of
  {\em Lecture Notes in Mathematics}, pages 53--69. Springer, 2001-2002.

\bibitem{BorellLyapunov}
Ch. Borell.
\newblock Complements of {L}yapunov's inequality.
\newblock {\em Math. Ann.}, 205:323--331, 1973.

\bibitem{Borell-logconcave}
Ch. Borell.
\newblock Convex measures on locally convex spaces.
\newblock {\em Ark. Mat.}, 12:239--252, 1974.

\bibitem{DafnisPaouris}
N.~Dafnis and G.~Paouris.
\newblock Small ball probability estimates, {$\psi_2$}-behavior and the
  hyperplane conjecture.
\newblock {\em J. Funct. Anal.}, 258(6):1933--1964, 2010.

\bibitem{FleuryOnVarianceConjecture}
B.~Fleury.
\newblock Between {P}aouris concentration inequality and variance conjecture.
\newblock {\em Ann. Inst. Henri Poincar\'e Probab. Stat.}, 46(2):299--312,
  2010.

\bibitem{FleuryImprovedThinShell}
B.~Fleury.
\newblock Concentration in a thin euclidean shell for log-concave measures.
\newblock {\em J. Func. Anal.}, 259:832--841, 2010.

\bibitem{FleuryGuedonPaourisCLP}
B.~Fleury, O.~Gu{\'e}don, and G.~Paouris.
\newblock A stability result for mean width of $l_p$-centroid bodies.
\newblock {\em Advances in Mathematics}, 214(2):865--877, 2007.

\bibitem{Fradelizi-Habilitation}
M.~Fradelizi.
\newblock Contributions \`a la g\'eom\'etrie des convexes. {M}\'ethodes
  fonctionnelles et probabilistes. {H}abilitation \`a {D}iriger des
  {R}echerches de l'{U}niversit\'e {P}aris-{E}st {M}arne {L}a {V}all\'ee, 2008.
\newblock http://perso-math.univ-mlv.fr/users/fradelizi.matthieu/pdf/HDR.pdf.

\bibitem{GardnerSurveyInBAMS}
R.~J. Gardner.
\newblock The {B}runn-{M}inkowski inequality.
\newblock {\em Bull. Amer. Math. Soc. (N.S.)}, 39(3):355--405, 2002.

\bibitem{GrunbaumSymmetry}
B.~Gr{\"u}nbaum.
\newblock Partitions of mass-distributions and of convex bodies by hyperplanes.
\newblock {\em Pacific J. Math.}, 10:1257--1261, 1960.

\bibitem{HaberlLpIntersectionBodies}
C.~Haberl.
\newblock {$L_p$} intersection bodies.
\newblock {\em Adv. Math.}, 217(6):2599--2624, 2008.

\bibitem{KLS}
R.~Kannan, L.~Lov{\'a}sz, and M.~Simonovits.
\newblock Isoperimetric problems for convex bodies and a localization lemma.
\newblock {\em Discrete Comput. Geom.}, 13(3-4):541--559, 1995.

\bibitem{KlartagPerturbationsWithBoundedLK}
B.~Klartag.
\newblock On convex perturbations with a bounded isotropic constant.
\newblock {\em Geom. and Funct. Anal.}, 16(6):1274--1290, 2006.

\bibitem{KlartagCLP}
B.~Klartag.
\newblock A central limit theorem for convex sets.
\newblock {\em Invent. Math.}, 168:91--131, 2007.

\bibitem{KlartagCLPpolynomial}
B.~Klartag.
\newblock Power-law estimates for the central limit theorem for convex sets.
\newblock {\em J. Funct. Anal.}, 245:284--310, 2007.

\bibitem{KlartagUnconditionalVariance}
B.~Klartag.
\newblock A {B}erry-{E}sseen type inequality for convex bodies with an
  unconditional basis.
\newblock {\em Probab. Theory Related Fields}, 45(1):1--33, 2009.

\bibitem{Ledoux-Book}
M.~Ledoux.
\newblock {\em The concentration of measure phenomenon}, volume~89 of {\em
  Mathematical Surveys and Monographs}.
\newblock American Mathematical Society, Providence, RI, 2001.

\bibitem{LutwakZhang-IntroduceLqCentroidBodies}
E.~Lutwak and G.~Zhang.
\newblock Blaschke-{S}antal\'o inequalities.
\newblock {\em J. Differential Geom.}, 47(1):1--16, 1997.

\bibitem{EMilman-Gaussian-Marginals}
E.~Milman.
\newblock On gaussian marginals of uniformly convex bodies.
\newblock {\em J. Theoret. Prob.}, 22(1):256--278, 2009.

\bibitem{EMilman-RoleOfConvexity}
E.~Milman.
\newblock On the role of convexity in isoperimetry, spectral-gap and
  concentration.
\newblock {\em Invent. Math.}, 177(1):1--43, 2009.

\bibitem{VMilman-DvoretzkyTheorem}
V.~D. Milman.
\newblock A new proof of {A}. {D}voretzky's theorem on cross-sections of convex
  bodies.
\newblock {\em Funkcional. Anal. i Prilo\v zen.}, 5(4):28--37, 1971.

\bibitem{Milman-Pajor-LK}
V.~D. Milman and A.~Pajor.
\newblock Isotropic position and interia ellipsoids and zonoids of the unit
  ball of a normed $n$-dimensional space.
\newblock In {\em Geometric Aspects of Functional Analysis}, volume 1376 of
  {\em Lecture Notes in Mathematics}, pages 64--104. Springer-Verlag,
  1987-1988.

\bibitem{Milman-Schechtman-Book}
V.~D. Milman and G.~Schechtman.
\newblock {\em Asymptotic theory of finite-dimensional normed spaces}, volume
  1200 of {\em Lecture Notes in Mathematics}.
\newblock Springer-Verlag, Berlin, 1986.
\newblock With an appendix by M. Gromov.

\bibitem{Paouris-Small-Diameter}
G.~Paouris.
\newblock $\psi_2$-estimates for linear functionals on zonoids.
\newblock In {\em Geometric Aspects of Functional Analysis}, volume 1807 of
  {\em Lecture Notes in Mathematics}, pages 211--222. Springer, 2001-2002.

\bibitem{Paouris-IsotropicTail}
G.~Paouris.
\newblock Concentration of mass on convex bodies.
\newblock {\em Geom. Funct. Anal.}, 16(5):1021--1049, 2006.

\bibitem{PaourisSmallBall}
G.~Paouris.
\newblock Small ball probability estimates for log-concave measures.
\newblock To appear in Trans. Amer. Math. Soc., 2010.

\end{thebibliography}

\def\cprime{$'$}

\end{document}